\documentclass[10pt,a4paper]{article} %modify affilation?!!!! Federal State Budgetary Educational Institution of Higher Education "Saint-Petersburg State University?!!
\usepackage{amsmath,amssymb,amsthm,amsfonts,amscd,euscript,verbatim, t1enc, newlfont}
\usepackage{hyperref}
\usepackage{graphicx}
\usepackage[all]{xy}
\providecommand{\keywords}[1]{\textbf{\textit{Key words and phrases }} #1}
\providecommand{\subjclass}[1]{\textbf{\textit{2020 Mathematics Subject Classification.}} #1}

\usepackage{color}
\usepackage{pgfplots}
\pgfplotsset{width=7cm, compat=1.10}
\usepgfplotslibrary{fillbetween}
\usepackage{caption}

\theoremstyle{definition}
\newtheorem{theo}{Theorem}[subsection]

\newtheorem{theore}{Theorem}[section]
\newtheorem{pr}[theo]{Proposition}
\newtheorem{prop}[theore]{Proposition}

 \newtheorem{lem}[theo]{Lemma}

\theoremstyle{remark}
\newtheorem{rema}[theo]{Remark}

\theoremstyle{definition}
\newtheorem{defi}[theo]{Definition}

\numberwithin{equation}{subsection}

\newcommand\ca{{\mathcal{A}}}
\newcommand\cp{\mathcal{P}}
\newcommand\cu{\underline{C}}
\newcommand\du{\underline{D}}

\newcommand\au{\underline{A}}
\newcommand\bu{\underline{B}}
\newcommand\hf{{\underline{HF}}}
\newcommand\wstu{w_{st}}
\newcommand\tst{t^{st}}
\newcommand\kw{K_{\mathfrak{w}}}

\newcommand\chow{\operatorname{Chow}}
\newcommand\chowm{\mathfrak{Chow}}
\newcommand\dm{DM}

\newcommand\q{{\mathbb{Q}}}

\newcommand\obj{\operatorname{Obj}}
\newcommand\hw{{\underline{Hw}}}

\newcommand\ab{\operatorname{Ab}}

\newcommand\ql{{\mathbb{Q}_{\ell}}}

\newcommand\zop{\mathbb{Z}[1/p]}%\big[1/p\big]}}%\frac{1}{p}\big]}}

\newcommand\z{{\mathbb{Z}}}

\newcommand\ii{\mathcal{I}}

\newcommand\wchow{w_{\chow}}
\newcommand\hwchow{\hw_{\chow}}

\newcommand\grwc{Gr_{W}}

\newcommand\ns{\{0\}}

\newcommand\mgcq{M_{gm}^{c,\q}}

\newcommand\mgcr{M_{gm}^{c,R}}

\newcommand\spe{\operatorname{Spec}}

\newcommand\com{\mathbb{C}}

\newcommand\id{\operatorname{id}}
 \newcommand\lan{\langle}
\newcommand\ra{\rangle}
\newcommand\bl{\bigl(} \newcommand\br{\bigl)}
\newcommand\var{\operatorname{Var}}
\newcommand\schpr{\operatorname{SchPr}}
\newcommand\sv{\operatorname{SmVar}}
\newcommand\spv{\operatorname{SmPrVar}}

\DeclareMathOperator\kar{\operatorname{Kar}}

\DeclareMathOperator\co{\operatorname{Cone}}
\DeclareMathOperator\prli{\varprojlim}
\DeclareMathOperator\inli{\varinjlim}

\newcommand\hu{\underline{H}}

\newcommand\dmerb{DM^{eff}_{R}}
\newcommand\dmrb{DM_{R}}
\newcommand\dmrbl{DM_{R}(K\perf)}

 \newcommand\dmcdh{\dm_{cdh}}

\newcommand\dmrr{DM^{r}_{R}{}}

\newcommand\mgr{M_{R}}

\newcommand\mgrc{M^{c}_{R}}

\newcommand\chower{\underline{\chow}^{eff}_R}

\newcommand\chowgri{\operatorname{CH}}
\newcommand\chw{\operatorname{CWH}}

\newcommand\dmger{{DM^{\,\, {eff}}_{\scalebox{0.7}{gm,R}}}{}}

\newcommand\dmer{DM_{R}^{eff}{}}
\newcommand\dmerpl{DM_{R\ \wchow+}^{eff}}
\newcommand\dmerm{DM_{R-}^{eff}{}}
\newcommand\dmeq{DM_{\q}^{eff}{}}

\newcommand\dmgep%{DM^{eff}_{gm,\zop}}  %
{{DM^{\,\, {eff}}_{\scalebox{0.6}{gm},\zop}}{}}

\newcommand\dmgepr %{DM^{eff,R}_{gm,\zop}} %
{{DM^{\,\, {eff}}_{\scalebox{0.5}{gm},\zop,R}}{}}
\newcommand\dmgepq %{DM^{eff,\q}_{gm,\zop}} %
{{DM^{\,\, {eff}}_{\scalebox{0.5}{gm},\zop,\q}}{}}

\newcommand\thomr{t_{hom}^R}

\newcommand\choweq{\underline{\chow}^{eff}_\q}

\newcommand\dmgeq{DM^{eff}_{gm,\q}{}}
\newcommand\perf{{}^{perf}}

\newcommand\dmr{\dm_R}
\newcommand\dmrwp{\dm_{R,w+}^{eff}{}}

\newcommand\afo{\mathbb{A}^1}

\begin{document}

\title{On Chow-weight homology of motivic complexes and its relation to motivic homology} 
 \author{Mikhail V. Bondarko, David Z. Kumallagov   \thanks{ %%!!!!
 %The main results of the paper were  obtained under 
The work of the authors on section 3 of the paper was supported %of
 by  the Russian Science Foundation,  grant no. 16-11-10200. %\newline
 %The %reported study of the 
% The work of the  second author on 
Section  1 and 2 were written by D. Kumallagov; this work was funded by RFBR, project number 19-31-90074.}}

\maketitle
\begin{abstract}
In this paper we study in detail the so-called Chow-weight homology of Voevodsky motivic complexes and relate it to motivic homology. %This allows us to
 We  generalize earlier results and  prove that the vanishing of higher motivic homology groups of  a motif  $M$ implies similar vanishing for its Chow-weight homology along with effectivity properties of the higher terms of  its weight complex $t(M)$ and of  higher Deligne weight quotients of its cohomology.  
Applying this statement to motives with compact support we obtain a similar relation between the vanishing of Chow groups and the cohomology with compact support of %arbitrary 
 varieties. Moreover, we prove that if higher motivic homology groups of a geometric motif  or a variety over a universal domain are torsion (in a certain "range") then the exponents of these groups are uniformly bounded. 

%The results of the current text allow to
 To prove our main results we study Voevodsky slices of motives. Since the slice functors do not respect the compactness of motives, the results of  the previous Chow-weight homology paper are not sufficient for our purposes; %so we 
 this is our main reason to extend them to ($\wchow$-bounded below) motivic complexes.
\end{abstract}


\tableofcontents

\section*{Introduction}

This paper is a continuation of  \cite{bsoscwhn} (see also the slightly different \cite{bsoscwho}). In these papers an extension of the well-known decomposition of the diagonal theory\footnote{Recall that this theory originates from \cite{blect}. Some of it was recalled and discussed in Propositions 0.4, 4.3.1, and 4.3.4, and Remarks %Remark 
0.5(1), %Proposition 4.3.1, 
3.3.9(1) and %Remark 
4.3.2(2) %, and Proposition 4.3.4
  of \cite{bsoscwhn}.} to geometric Voevodsky motives and varieties was proposed; its  main  tool were the  new {\it Chow-weight homology} theories. Since the main purpose of %that paper
 these papers  was the study of varieties, %only  the 
 Chow-weight homology   was only defined on the categories of %$\dmger$ of effective 
 geometric motives.

In the current text % paper 
 we demonstrate that it makes sense to study Chow-weight homology of objects of the bigger category $\dmer$ of %(unbounded) 
 $R$-linear motivic complexes as well. So, we extend the main results of %ibid.
  \cite{bsoscwhn}  to the subcategory $\dmerpl$ of $\dmer$ (here $R$ is the coefficient ring) that consists of objects that are bounded below with respect to the {\it Chow weight structure}. This enables us to generalize Corollary 3.4.2  of  \cite{bsoscwhn} and prove that the vanishing of higher motivic homology groups of a motif $M$ is equivalent to similar vanishing for Chow-weight homology %along with 
of $M$	 and also to the corresponding effectivity properties of higher terms of the Chow-weight complex $t(M)$ of $M$. % and of higher Deligne weight quotients.  
 The difference with loc. cit. is that we are able to treat motivic homology of positive dimensions (that correspond to complexes of algebraic cycles of dimension $j>0$) in the corresponding Theorem \ref{main} below. The proof of that theorem uses Voevodsky slices of motives; thus it is necessary to consider Chow-weight homology and conditions related to it %CORRECTION
 for objects of $\dmer$  that are not geometric.

Next we apply Theorem \ref{main} to extend some more results of ibid. We prove that if higher motivic homology groups of a geometric motif  $M$ %or a variety 
 over a universal domain are torsion (in a certain "range") then the exponents of these groups are uniformly bounded. 
Moreover,  we apply Theorem \ref{main} to motives with compact support of varieties (cf. Corollary 4.2.3 of  ibid.); we obtain  that if certain Chow groups of a variety $X$ over a universal domain $K$ are torsion then they are of bounded exponent. 
 %Next, 
Arguing similarly to Theorem 4.2.1 ibid. we also obtain that %relation between
 %the vanishing of 
 %if certain Chow groups of a variety $X$ over a universal domain $K$ are torsion then they are of bounded exponent, 
 this torsion assumption implies certain effectivity conditions for the  cohomology of $X$ with compact support; see Theorem \ref{tmgc}. % of arbitrary varieties. Moreover, 

%dedicated to the study of 

Now we describe the contents of the paper. More details can be found at the beginnings of sections.

In \S\ref{sprel} we recall %some preliminary statements on 
 several properties of (smashing) weight structures, Voevodsky's motivic category $\dmer$ and its localizations,  and Chow weight structures on them.

In \S\ref{scwh} we %extend certain results of \cite{bsoscwhn} and 
 define Chow-weight homology  on $\dmer$, and extend its properties (as studied in \cite{bsoscwhn}) %of Chow-weight homology
 to the subcategory $\dmerpl\subset \dmer$ of  Chow-bounded below objects. % of  $\dmer$. 
 This enables us to prove the main Theorem \ref{main}; it says  that the vanishing of higher motivic homology groups of an object $M$ of $\dmerpl$ (over all function field extension of the base field $k$) is equivalent to the similar vanishing for Chow-weight homology of $M$. %it 
 This vanishing is also equivalent to certain effectivity assumptions %along with effectivity of higher terms of Chow-weight complexes.
on  the weight complex $t(M)$ (that is, its higher terms should be "big Chow motives" that are "effective enough"); moreover, it has a "coproduct and extension-closure" re-formulation. 

%the category
%along with its relation to motivic homology to 

In \S\ref{sappl} we apply Theorem \ref{main} to geometric motives and combine it with the results and arguments of ibid. Firstly we prove 
that if higher motivic homology groups of a geometric motif  over a universal domain are torsion (in a certain "range") then the exponents of these groups along with the related Chow-weight homology ones are uniformly bounded. Next, we combine earlier results with the properties of motives with compact support to obtain the aforementioned  Theorem \ref{tmgc}. 

In %\S
 Appendix \ref{sapp} we prove some properties of motives that are necessary both for the current paper and for \cite{bsoscwhn}. They appear to be well-known even though the authors were not able to find them in the literature.

\section{Preliminaries}\label{sprel}

In \S\ref{snotata} we recall some %categorical 
 definitions; they are mostly related to  ({\it smashing}) triangulated categories.

In \S\ref{smot} we recall some basics on ($R$-linear) Voevodsky motives over a perfect field $k$. % (with coefficients in a ring

In \S\ref{ssws} we recall  basic definitions and statements on weight structures. 

In \S\ref{sexw} we discuss {\it purely compactly generated} weight structures and the weight structures they induce on ("purely compactly generated") localizations.

In \S\ref{swc} we recall some of the theory of {\it strong weight complexes}, {\it pure} (homological) functors and weight spectral sequences.

In \S\ref{swchow} we apply the general theory to the category $\dmer$ and its localizations; this gives certain Chow weight structures whose hearts are "generated" by Chow motives.

\subsection{Some notation and conventions}\label{snotata}

\begin{itemize}
\item For $a\le b\in \z$ we will write $[a,b]$ (resp. $[a,+\infty)$, resp.  $[a,+\infty]$) for the set $\{i\in \z:\ a\le i \le b\}$ (resp. $\{i\in \z:\ i\ge a \}$, resp. $[a,+\infty)\cup \{+\infty\}\subset \z\cup  \{+\infty\}$); we will never consider  real line segments in this paper. Respectively, when we write $i\ge c$ (for $c\in \z$) we mean that $i$ is an integer satisfying this inequality.%CORRECTION

\item Given a category $C$ and  $X,Y\in\obj C$  we will write $C(X,Y)$ for  the set of morphisms from $X$ to $Y$ in $C$.

\item For categories $C',C$ we write $C'\subset C$ if $C'$ is a full %strict
subcategory of $C$.

\item Given a category $C$ and  $X,Y\in\obj C$, we say that $X$ is a {\it
retract} of $Y$ %(and $Y$ is {\it coretract} of $X$)
 if $\id_X$ can be %factorized as $X\stackrel{i}{\to} Y\stackrel{p}{\to}X$
 factored through $Y$.\footnote{Clearly, if $C$ is triangulated or abelian, %{del}
then $X$ is a retract of $Y$ if and only if $X$ is its direct summand.}\

\item  Let  %n additive?? 
 $\hu$ be a subcategory of an additive category $C$ 
%$\hu\subset C$ a subcategory $\hu$ is called

 Then $\hu$ is said to be {\it retraction-closed}   in $C$ if it contains all retracts of its objects in $C$.%CORRECTION
Moreover,  the full subcategory $\kar_{C}(\hu)\subset C$ %of and additive category $C$ 
 whose objects are the retracts of objects of a subcategory $\hu$ (in $C$) will be called the {\it retraction-closure} of $\hu$ in $C$. 
\end{itemize}
 
Let us now list some less common definitions and conventions. Below  the symbol $\cu$ below will always denote some triangulated category;
%not always. ``unless stated otherwise''?
usually it will be endowed with a {\it weight structure} $w$ (see Definition \ref{dwstr} below). % (see \S1.2 of \cite{brelmot}).

\begin{defi}\label{dfir}
Let $\bu$ be an additive category.
\begin{enumerate}
\item\label{ifac} 
We call a category $\frac {\bu} {\hu}$ the {\it factor} %of an additive category $A$
 $\bu$ by its full additive subcategory $\hu$ if $\obj \bl \frac {\bu} {\hu} \br=\obj
\bu$ and $(\frac {\bu} {\hu})(X,Y)= \bu(X,Y)/(\sum_{Z\in \obj \hu} \bu(Z,Y) \circ \bu(X,Z))$.

\item %For an additive category $\bu$ we  
 We will write $K(\bu)$ for the homotopy category of (cohomological) complexes over $\bu$. %Its full subcategory of bounded complexes will be denoted by $K^b(\bu)$.??
  We will write $M=(M^i)$ if $M^i$ are the terms of the complex $M$.

\item For any  $A,B,C \in \obj\cu$ we  say that $C$ is an {\it extension} of $B$ by $A$ if there exists a distinguished triangle $A \to C \to B \to A[1]$.

\item A class $D\subset \obj \cu$ is said to be  {\it extension-closed}
    if it %??!!
     %for any distinguished triangle $X\to Y\to Z$ in $\cu$ we have the following: $X,Z\in D\implies Y\in D$. 
		is closed with respect to extensions and contains $0$. We  call the smallest extension-closed subclass 
of objects of $\cu$ that  contains a given class $B\subset \obj\cu$   the {\it extension-closure} of $B$. 

%Moreover, we will call  the smallest extension-closed retraction-closed subclass  of objects of $\cu$ that  contains $B$ the {\it envelope} of $B$. 

\item Given a class $D$ of objects of $\cu$ we will write $\lan D\ra$ or $\lan D\ra_{\cu}$ for the smallest full retraction-closed
triangulated subcategory of $\cu$ containing $D$. We   call  $\lan D\ra$  the triangulated category {\it densely generated} by $D$.

\item For $X,Y\in \obj \cu$ we  write $X\perp Y$ if $\cu(X,Y)=\ns$. For
$D,E\subset \obj \cu$ we write $D\perp E$ if $X\perp Y$ for all $X\in D,\
Y\in E$.
Given $D\subset\obj \cu$ we  will write $D^\perp$ for the class
$$\{Y\in \obj \cu:\ X\perp Y\ \forall X\in D\}.$$
%Sometimes we will denote by $D^\perp$ the corresponding full subcategory of $\cu$. 
Dually, ${}^\perp{}D$ is the class $\{Y\in \obj \cu:\ Y\perp X\ \forall X\in D\}$.

\item\label{ineg} %Let $\hu$ be a full subcategory of a triangulated category
 Assume that $\bu \subset \cu$. We  say that $\bu$ is {\it connective} (in $\cu$) if $\obj \bu\perp (\cup_{i>0}\obj (\bu[i]))$.

\item Given $f\in\cu (X,Y)$, where $X,Y\in\obj\cu$, we  call the third vertex
of (any) distinguished triangle $X\stackrel{f}{\to}Y\to Z$ a {\it cone} of
$f$.\footnote{Recall %{del}
that different choices of cones are connected by non-unique isomorphisms.}\

%\item Note yet that we will call any (covariant) homological functor (from a triangulated category) a homology theory. Consequently, for a complex $A=(A^i,d^i)$ of abelian groups we  call the quotient $\ke d^i/\imm d^{i-1}$ the $i$-th {\bf homology} of $A$. In particular, we use this "cohomological" convention for the Chow-weight homology theory; respectively, we will write $i$ as an upper index in the corresponding notation.

%On the other hand, if $H$ is a homological functor then we will use the notation $H_i$ for $H\circ [-i]$.
\end{enumerate}
\end{defi}

%All coproducts in this paper will be small.

Let us now list some definitions related to smashing triangulated categories.

\begin{defi}\label{desmash}
Assume that $\cp$ is a class of objects of $\cu$, and $\cu$ is {\it smashing}, that is, %$\cu$ is 
closed with respect to (small) coproducts.

\begin{enumerate}
\item\label{idsmash} 
We  say that %a subcategory 
 a class of objects or a full subcategory of $\cu$ is smashing (in $\cu$) if it is closed with respect to $\cu$-coproducts.

\item\label{idlocal}
We  say  that a full %strict %triangulated 
subcategory $\du\subset \cu$ is  {\it localizing}  whenever it is  triangulated and smashing. %closed with respect to $\cu$-coproducts. 

Respectively, we  call the smallest localizing  subcategory of $\cu$ that contains %a given class $\cp\subset \obj \cu$  
 $\cp$ the {\it  localizing subcategory of $\cu$ generated by $\cp$}.  % we will write $\lan \cp\ral$ (resp. $\lan \cp \rab$) for this subcategory.

\item\label{idhull}
If  $\hu$ is a %be a (not necessarily additive) 
subcategory of $\cu$ %an additive category $\bu$.We will call
 %of  the full additive subcategory of $\bu$ 
 then we  call the full subcategory of $\cu$ whose objects are %all retracts of %arbitrary (small) 
 the retracts of coproducts of objects of $\hu$ in $\cu$ the {\it coproductive hull} of $\hu$ (in $\cu$); we will use the notation $\hu^{\widehat\oplus}$ for it.

 \item\label{icompact} An object $M$ of $\cu$ is said to be {\it compact} if %and only if  if $\cu'$ is smashing and 
 the functor $H^M=\cu(M,-):\cu\to \ab$ respects coproducts. 

 \item\label{icompactg}  We  say that $\cu$ is {\it compactly generated} by $\cp$ if $\cp$ is %{\bf set} 
 an essentially small class of compact objects of $\cu$ that generates $\cu$ as its own localizing subcategory.
\end{enumerate}
\end{defi}

\subsection{On Voevodsky motivic complexes and certain localizations}\label{smot}

We start with some preliminaries and notation for motivic complexes.\\
In this section $k$ will denote a fixed perfect base field of characteristic $p$,
and we set $\zop=\z$ if $p=0$.

The set of smooth projective varieties over $k$ will be denoted by $\spv$.

\begin{itemize}

\item For a (fixed) unital commutative associative %ring (or 
 $\zop$-algebra $R$ %being %a fixed 
we consider the $R$-linear motivic categories $\dmger \subset \dmerb\subset \dmrb$ (see \cite[\S4]{bokum}). The categories $\dmerb$ and $\dmrb$  %is %closed with respect to small coproducts, % (and consequently, idempotent complete);
  are smashing (see Definition \ref{desmash}), and the embedding $\dmer\to \dmrb$ respects coproducts.  Moreover, $\dmerb$  is compactly generated by its triangulated subcategory $\dmger$ of effective geometric motives. % (see Definition \ref{desmash}). 
	
	%A basic part of the construction of motives 
	 \item There is a functor $\mgr$ %(of the $R$-motive)??
 ($R$-motif) from the category of smooth $k$-varieties into $\dmger$. Actually, $\mgr$ extends to the category of all $k$-varieties (see \cite{1} and \cite{kellyast}); yet we will  mention this extension just a few times.  We will write $R$ for the object $\mgr(\spe k)$.
 %never??????
	
	Moreover, $\dmger$ is  densely 
%actually compactly
  generated (see \S\ref{snotata}) by the $R$-linear motives $M_{R}(\spv)$ (see Theorem 2.1.2 of ibid.); hence the set $M_{R}(\spv)$  compactly generates $\dmerb$ as well. 

\item We will write $\chower$ % \subset\dmger$ % \subset \dmgmr$ 
 %will denote the category 
 for the Karoubi-closure in $\dmerb$ of the subcategory whose object class equals $M_{R}(\spv)$. % are $R$-motives of smooth projective varieties;  %We will call 
 $\chower$ will be called the category of $R$-linear effective homological Chow motives; see %Proposition \ref{pcrulemma}(\ref{iсru0})  below or??????!
 Remark 1.3.2(4) of \cite{bokum} for a  justification of this terminology.

\item We also introduce the following notation: $R \lan 1 \ra$ will denote the $R$-linear Lefschetz object (this is $R(1)[2]$ in the notation of \cite{1}). For $i \ge 0$ and $M \in \obj \dmerb$ we will write $M \lan i \ra$ for the object $M \otimes_{\dmerb}(R \lan 1 \ra)^{\otimes i}$. 
 
Recall that the functor $- \lan i \ra$ is a full embedding of $\dmerb$ into itself; thus the essential image $\dmerb \lan i \ra$ of this functor is a full subcategory of $\dmer$ that is equivalent to  $\dmerb$ itself.

Moreover, $- \lan 1 \ra$ extends to an %invertible 
 exact %functor on 
 autoequivalence of $\dmrb$, and the corresponding %objects 
 class $M_{R}(\sv)\lan i \ra$ consists of compact objects for any $i\in \z$.

\item Note that for any $i \ge 0$, $R \lan i \ra$ is a retract of $M_{R}((\mathbb{P}^1)^{i})$; thus $\chower \lan i \ra \subset \chower$.

%\item For $ m \in \mathbb{Z}$ we will write $d_{\leq m} \dmerb$ for the localizing subcategory of $\dmerb$ generated by $\{M_{R}(X)\}$ for X running through smooth $k$-varieties of dimension at most $m$. We note that $d_{\leq m} \dmerb$ is compactly generated by $\{M_{R}(P)\}$ for $P$ running through smooth projective $k$-varieties of dimension $\leq m$ (Remark 2.2.3 of \cite{bsoscwh}).

\end{itemize}

%Let us also give some (less common)
We will also need the following definitions %and properties of 
 related to motives.

\begin{defi}\label{dchowm}
Let $K/k$ be a field extension, and $M$ an object of $\dmerb$.

1. %If $K$ if a field then 
Then $K\perf$ will denote the perfect closure of $K$.

2. %For a field extension $K/k$ and $M\in \obj \dmer$ %a %motive 
 %motive $M$ over $K$ we %will denote its image with respect to the corresponding base change functor by 
  We will use the notation $M_K$ for %its image with respect to the corresponding base change functor.
the image of $M$ with  respect to the  base field change functor $\dmer\to \dmer(K\perf)$; % corresponding to the extension $L\perf/k$; 
 see appendix %\S
 \ref{sapp} below for some information on functors of this type.

3 For %$M$ being an object $M$ of $\dmerb$ and
  $l,j\in \z$ %, and $j\ge 0$
  we define $\chowm_{j}(M_K,R,l)$ (resp. $\chowm_{j}(M_K,R)$) as the group $\dmrbl(R\lan j \ra[l],M_K)$ (resp. $\dmrbl(R\lan j \ra,M_K)$).\footnote{In \cite{bsoscwhn} the group $\dmrbl(R\lan j \ra[l],M_K)$ is denoted by $h_{2j+l,j}(M_K,R) $.}

4. For $i\ge -1$  we will write $\dmr^{i}$ for the Verdier quotient $\dmerb/\dmerb \lan i+1 \ra$. $l^{i}$ will denote the corresponding localization functor, and $\mgr^i=l^i\circ \mgr$.
\end{defi}

\begin{pr}\label{pvan}
Let $j,l\in \z$,  $r\ge 0$, and assume $j-r+l<0$.
Then %the following statements are valid.
%1. Let $N\in \obj \chowr$.  Then  $$\chowm_{j}(N_K,R,l)\cong \dmgr(K)(\widehat{N}_K,R\lan j \ra[-l]) %(K\perf)(P_K(d-j)[2d-2j-2l])$ $$ for any perfect field $K/k$, where $\widehat{N}$ is the Poincare dual of $N$ (in $\chowr\subset \dmgmr$).  1. For 
 for any $N\in \obj \chower$ and any %perfect 
 field extension $K/k$ we have   $\chowm_{j}(N_K\lan r\ra,R,l)=\ns$. % if $j-r+l<0$. % and $K$ is any field extension of $k$.
\end{pr}

\begin{proof}
%1. This is just the Poincare duality for Voevodsky motives; see Theorem 5.23 of \cite{degdoc}.
This is an easy consequence of the well-known properties of (Suslin or Bloch) cycle complexes along with Proposition \ref{pmot}(\ref{ie1}) below;  see  %Theorem 5.3.14 of \cite{kellyast}.
  Proposition 2.3.3(2) of \cite{bsoscwhn}. 
\end{proof}

\subsection{Weight structures: basic definitions and statements}\label{ssws}

Let us recall the definition of the  notion that is central for this paper.

\begin{defi}\label{dwstr}

I. A pair of subclasses $\cu_{w\le 0},\cu_{w\ge 0}\subset\obj \cu$ %(of {\it $w$-negative} and {\it $w$-positive} objects, respectively)
will be said to define a weight
structure $w$ for a triangulated category  $\cu$ if 
they  satisfy the following conditions.

(i) $\cu_{w\ge 0}$ and $\cu_{w\le 0}$ are %additive and 
retraction-closed in $\cu$ (i.e., contain all $\cu$-retracts of their objects).

(ii) {\bf Semi-invariance with respect to translations.}

$\cu_{w\le 0}\subset \cu_{w\le 0}[1]$, $\cu_{w\ge 0}[1]\subset
\cu_{w\ge 0}$.

(iii) {\bf Orthogonality.}

$\cu_{w\le 0}\perp \cu_{w\ge 0}[1]$.

(iv) {\bf Weight decompositions}.

 For any $M\in\obj \cu$ there
exists a distinguished triangle
%\begin{equation}\label{wd}
$$X\to M\to Y%\stackrel{f}
{\to} X[1]$$
%\end{equation} 
%\end{equation} 
such that $X\in \cu_{w\le 0},\  Y\in \cu_{w\ge 0}[1]$.\end{defi}

We will also need the following definitions related to triangulated categories and weight structures.

\begin{defi}\label{dwso}

Let $i,j\in \z$; assume that a triangulated category $\cu$ is endowed with a weight structure $w$. %??, and $\cu$ is another triangulated category.

\begin{enumerate}
\item\label{idh}
The full category $\hw\subset \cu$ whose objects class is 
$\cu_{w=0}=\cu_{w\ge 0}\cap \cu_{w\le 0}$ is called the {\it heart} of $w$.

\item\label{id=i}
 $\cu_{w\ge i}$ (resp. $\cu_{w\le i}$, resp.
$\cu_{w= i}$) will denote $\cu_{w\ge
0}[i]$ (resp. $\cu_{w\le 0}[i]$, resp. $\cu_{w= 0}[i]$).

\item\label{id[ij]}
$\cu_{[i,j]}$  denotes $\cu_{w\ge i}\cap \cu_{w\le j}$; clearly this class  equals $\ns$ if $i>j$. 

%??$\cu^b\subset \cu$ will be the category whose object class is $\cup_{i,j\in \z}\cu_{[i,j]}$.

%\item\label{idbo} We will  say that $(\cu,w)$ is {\it  bounded}  if $\cu^b=\cu$ (i.e., if $\cup_{i\in \z} \cu_{w\le i}=\obj \cu=\cup_{i\in \z} \cu_{w\ge i}$).

\item \label{lrbo} We will call $\cu_{w+}=\cup_{i\in \z} \cu_{w\ge i}$ %(resp. $\cup_{i\in \z} \cu_{w\le i}$) 
the class of $w-$bounded below %(resp., $w-$bounded above) 
objects of $\cu$. % (they will be denoted as $\cu_{+}$ and $\cu_{-}$ respectively).

\item\label{smash}
We will say that $w$ is {\it smashing} if  $\cu$ is smashing and the class $\cu_{w\ge 0}$ is %closed with respect to $\cu$-coproducts 
	  smashing (in it;  see Definition \ref{desmash}(\ref{idsmash} and cf.  Proposition  \ref{pbw}(\ref{leftsm}) below). %?????

\item\label{idwe}

	 Assume that a triangulated category $\cu'$ is endowed with a weight structures $w'$; let $F:\cu\to \cu'$ be an exact functor.

$F$ is said to be  {\it  weight-exact} (with respect to $w,w'$) if it maps $\cu_{w\le 0}$ into $\cu'_{w'\le 0}$ and
sends $\cu_{w\ge 0}$ into $\cu'_{w'\ge 0}$. 

\item\label{idrest}
Let $\du$ be a full triangulated subcategory of $\cu$.

We will say that $w$ {\it restricts} to $\du$ whenever the couple $(\cu_{w\le 0}\cap \obj \du,\ \cu_{w\ge 0}\cap \obj \du)$ is a weight structure on $\du$.

\item\label{degen}
We will say that $M$ is left (resp., right) $w$-degenerate if $M$ belongs to $\cap_{i \in \mathbb{Z}}\cu_{w\ge i}$ (resp. to $\cap_{i \in \mathbb{Z}}\cu_{w\le i}$).\\
Accordingly, $w$ is left (resp., right) non-degenerate if all left (resp. right) weight-degenerate objects are zero.
	\end{enumerate}

\end{defi}

\begin{rema}\label{rstws}

1. A  simple (and yet quite useful) example of a weight structure comes from the stupid
filtration on %$K^b(\bu)$ (or on 
$K(\bu)$ for an arbitrary additive category $\bu$.
In this case $K(\bu)_{w\le 0}$ (resp. $K(\bu)_{w\ge 0}$) is the class of complexes that are
homotopy equivalent to complexes  concentrated in degrees $\ge 0$ (resp. $\le 0$); see \cite[Remark 1.2.3(1)]{bonspkar}; this weight structure will be denoted by $w^{st}$.
 
 The heart of this weight structure is the retraction-closure  of $\bu$ %in  $K^b(\bu)$ (or 
in $K(\bu)$ (see \S\ref{snotata}). %, respectively). 

2. A weight decomposition (of any $M\in \obj\cu$) is (almost) never canonical. 

Still for any $m\in \z$ the axiom (iv) gives the existence of distinguished triangle \begin{equation}\label{ewd} w_{\le m}M\to M\to w_{\ge m+1}M \end{equation}  with some $ w_{\ge m+1}M\in \cu_{w\ge m+1}$ and $ w_{\le m}M\in \cu_{w\le m}$; we will call it an {\it $m$-weight decomposition} of $M$.

 We will often use this notation below (even though $w_{\ge m+1}M$ and $ w_{\le m}M$ are not canonically determined by $M$);
we will call any possible choice either of $w_{\ge m+1}M$ or of $ w_{\le m}M$ (for any $m\in \z$) a {\it weight truncation} of $M$.
Moreover, when we will write arrows of the type $w_{\le m}M\to M$ or $M\to w_{\ge m+1}M$ we will always assume that they come from some $m$-weight decompositions. % of $M$.  

3. In the current paper we use the ``homological convention'' for weight structures; 
it was previously used in \cite{wild} %\cite{bgern},    \cite{bkillw}, 
 %\cite{binters},  \cite{bonspkar},  %\cite{bpgws}, 	 \cite{bgn},
	and %in \cite{bwcomp}
	 in several papers of the authors, whereas in 
\cite{bws} %and in \cite{bger} 
 the ``cohomological convention'' was used. In the latter convention 
the roles of $\cu_{w\le 0}$ and $\cu_{w\ge 0}$ are interchanged, i.e., one
considers   $\cu^{w\le 0}=\cu_{w\ge 0}$ and $\cu^{w\ge 0}=\cu_{w\le 0}$. 
 
 We also recall that D. Pauksztello has introduced weight structures independently in \cite{konk}; he called them co-t-structures.

 %4. The orthogonality axiom (iii) in Definition \ref{dwstr} immediately yields that $\hw$ is connective in $\cu$ in the sense specified in  Definition \ref{dfir}(\ref{ineg}).  We will formulate a certain converse to this statement  in Proposition \ref{pexw} below.
 \end{rema}

\begin{pr} \label{pbw}
Let $\cu$ be a triangulated category, $n\ge 0$; we will assume 
that $w$ is a fixed 
weight structure on $\cu$.

\begin{enumerate}
\item\label{leftsm}
 $\cu_{w\le 0}$ is closed with respect to all coproducts that exist in $\cu$.

 \item\label{iort}
 $\cu_{w\ge 0}=(\cu_{w\le -1})^{\perp}$ and $\cu_{w\le 0}={}^{\perp} \cu_{w\ge 1}$.

%\item\label{i01}  $\cu_{[0,1]}$ consists exactly of cones of morphisms in $\hw$ (in $\cu$).

%\item\label{iwe} Assume that $w'$ is a weight structure on a triangulated category $\cu'$. Then an exact functor $F:\cu\to \cu'$ is weight-exact if and only if $F(\cu_{w=0})\subset \cu'_{w'=0}$.

 \item\label{iwd0} %For any weight decomposition of an
 If $M$ belongs to $ \cu_{w\ge -n}$ %and $m\ge 0$ %(see (\ref{wd})) 
 %we have (see Remark \ref{rstws}(3)) 
 then $w_{\le 0}M$ belongs to $ \cu_{[-n,0]}$.

%\item\label{itrun} If $m<l\in \z$ and $M\in \obj \cu$ then for any choice of arrows $w_{\le l}M\to M$ and $w_{\le m}(w_{\le l}M)\to w_{\le l}M$ that can be completed to an $l$-weight decomposition and an $m$-weight decomposition triangle (see Remark \ref{rstws}(2)) respectively,  the composition morphism   $w_{\le m}(w_{\le l}M)\to M$ can be completed to an  $m$-weight decomposition of $M$.

\item\label{ifactp} Assume that  $\du\subset \cu$ is a triangulated subcategory of $\cu$ such that $w$ restricts to a weight structure $w_{\du}$ on $\du$. Let $M\in \cu_{w \ge 0}$, $N\in \cu_{w=0}$, and assume that a morphism $f\in \cu(N,M)$ vanishes in the localization $\cu/\du$. 

Then $f$ factors through some object of $\hw_{\du}$.

%\item\label{ifactps} 
%Let $\du$ be a (full) triangulated subcategory of $\cu$ such that $w$ restricts to $\du$; % (i.e., the classes $\obj \du\cap \cu_{w\le  0}$ and $\obj \du\cap \cu_{w\ge  0}$ give a weight structure on $\du$); 
%let $M\in \cu_{w\le 0}$, $N\in \cu_{w\ge -n}$, %for some $n\ge 0$,
%and $f\in \cu(M,N)$. Suppose that $f$ factors through an object $P$ of $\du$, i.e., there exist $u_1 \in\cu(M,P)$ and $u_2 \in \cu(P,N)$ such that $f=u_2 \circ u_1$. 
%Then $f$ factors through an element of $\du_{[-n,0]}$.
\end{enumerate}
\end{pr}
\begin{proof}
%All these statements were 
 Assertions \ref{iort}--\ref{iwd0} were proved in \cite{bws} (pay attention to Remark \ref{rstws}(3)!). Assertion \ref{ifactp} is given by Corollary 1.4.6(2) of \cite{bsoscwhn}.
\end{proof}

\subsection{Some existence of weight structures statements}\label{sexw}  %On weight structures in localizations}

%Below we will consider certain {\it purely compactly generated} weight structures. So we recall the corresponding statements.
 
\begin{pr}\label{pexw}
Let $\bu$ be a connective additive subcategory of  a smashing (triangulated) $\cu$, and assume that objects of $\bu$ are compact in $\cu$.

Then there exists a %unique smashing 
weight structure $w$ on $\cu$ such that  $\cu_{w\le 0}$ (resp. $\cu_{w\ge 0}$) is the smallest subclass of $\obj \cu$ that is closed with respect to coproducts, extensions, and contains $\obj \bu[i]$ for $i\le 0$ (resp. for $i\ge 0$). Moreover, %$w$ is generated by $\obj \bu$ (in the sense of Definition \ref{dwso}(\ref{iwgen})) and????
  %$\cu_{w=0}$ is the class of $\cu$-retracts of all small $\cu$-coproducts of objects of $\bu$.
	$\hw=\bu^{\widehat\oplus}$ (see Definition \ref{desmash}(\ref{idhull})).
	
	In this case we will say that $w$ is { purely compactly generated} by $\bu$.
	
	Furthermore, if the objects of $\bu$ (compactly) generate $\cu$ as its own localizing subcategory then $w$ is left non-degenerate.
\end{pr}
\begin{proof}
The statement easily follows from Corollary  2.3.1 and Lemma 2.3.3 of \cite{bsnew}; cf. Theorem 3.2.2(2,3) of \cite{bwcomp}.
\end{proof}

Now let us discuss certain weight structures in localizations.

\begin{pr}\label{winloc}
Assume that $(\cu,w,\bu)$ are as in the previous proposition; in addition,  $\bu$  is  essentially small and  generates $\cu$ as its own localizing subcategory, and   $\hu$ is an  additive subcategory of $\bu$. % that
   %is closed with respect to coproducts. Let $B \subset \obj \cu$ is a set (i.e., it is small) and  let $\hu \subset \cu$ be a small additive category, objects in which are compact and generates $\cu$ as its own localizing subcategory.
 Denote by $\du$ the localizing subcategory of $\cu$ generated by $\hu$. Then the following statements are valid.

1. The Verdier quotient category $\cu/\du$ exists (i.e., it is a locally small category); the localization functor $\pi: \cu \to \cu/\du$ respects  coproducts and converts compact objects into compact ones.
%and possesses a right adjoint $G$ that is a full embedding functor. Besides, 
 Moreover, $\cu/\du$ is generated by $\pi(\obj \bu)$ as its own localizing subcategory, and the corresponding %localization 
 exact functor  %$\lan\hu\ra_{\cu}\to \lan\bu\ra_{\cu}$ 
$\lan\bu\ra_{\cu}/\lan\hu\ra_{\cu}\to \cu/\du$ (where  $\lan\bu\ra_{\cu}/\lan\hu\ra_{\cu}$ is the Verdier quotient of the corresponding locally small categories) is a full embedding.%CORRECTIONS

%b.) Assume moreover that $\hu$ is connective. Then $\cu$ possesses a weight structure $w$, whose heart is equivalent to the coproductive hull of $\hu$ (see Definition \ref{desmash}(\ref{idhull})).

2. $\cu/\du$ possesses a weight structure $w_{\cu/\du}$ such that $\pi$ is weight-exact. Moreover, $w_{\cu/\du}$ is purely compactly generated by  its full subcategory  corresponding to $\bu$ (in the sense of Proposition \ref{pexw}), and the corresponding functor $\hw\to \hw_{\cu/\du}$ factors as the composition of  the obvious functor $\hw\to \hw/\hu^{\widehat\oplus}$ (see Definition \ref{dfir}(\ref{ifac})) with a full embedding.
%2%. The restrictions of $\pi$ to $\cu_{+}, \cu_{-}$, and $\cu^{b}$, respectively, yield full embeddings into $\cu/\du$ of the Verdier localizations $\cu_{+}, \cu_{-}$, and $\cu^{b}$, respectively.
\end{pr}

\begin{proof}
 All these assertions were proved in \cite{bos} (see Proposition 4.3.1.3(III) and Theorem 4.3.1.4 of ibid.).
%2. This is exactly the Proposition 3.2.1(2) of \cite{bsnew}.
\end{proof}

\subsection{On weight complexes, pure functors, and weight spectral sequences}\label{swc} 
\label{swss}

Now we recall the theory of so-called "strong" weight complex functors. Note here that this version of the theory is less general than the "weak" one that was used in \cite{bsoscwhn}. The latter one is sufficient for our purposes (and is somewhat more convenient for them); yet it requires some non-standard definitions.

\begin{pr}\label{pwt}
%Assume that $\cu$ is endowed with a weight structure $w$. % and $h:M'\to \co(g)$ is the second side of a distinguished triangle containing $g$.
 Assume that $\cu$ possesses an $\infty$-enhancement (see \S1.1 of \cite{sosnwc} for the corresponding references), and  satisfies the assumptions  of  Proposition \ref{pexw}. %??; let  $\ca$ be an additive covariant %(resp., contravariant)
  %functor from $\hw$ into an abelian category $\au$.
%is endowed with a  weight stru
%bounded weight structure $w$.

Then there exists an exact functor $t^{st}:\cu\to K(\hw)$, $M\mapsto (M^i),$  such that the following statements are fulfilled. %that enjoys the following properties. 

\begin{enumerate}
\item\label{iwcbase} The composition of the embedding $\hw\to \cu$ with $t^{st}$ is isomorphic to the obvious embedding $\hw \to K(\hw)$.

%\item\label{irwcsh} Let $n\in \z$. Then $t\circ [n]_{\cu}\cong [n]_{\kw(\hw)}\circ t$, where  $[n]_{\kw(\hw)}$ is the obvious shift by $[n]$ (invertible) endofunctor of the category $\kw(\hw)$.

 \item\label{iwcfunct} Let $\cu'$ be a triangulated category that possesses an $\infty$-enhancement as well and is endowed with a compactly purely generated weight structure $w'$; let  $F:\cu\to \cu'$ be a weight-exact functor that lifts to $\infty$-enhancements. Then the composition $t'{}^{st}\circ F$ is isomorphic to $K(\hf)\circ t{}^{st}$, where 
$t'{}^{st}$ is the weight complex functor corresponding %to the restriction of ?!
 $w'$, and the functor $K(\hf):K(\hw)\to K(\hw')$ is the obvious $K(-)$-version of the restriction $\hf:\hw\to \hw'$ of $F$.

\item\label{iwcw}
%For any 
 Fix a choice of weight truncations $w_{\le i}N$ of $N\in \obj \cu$ (see Remark \ref{rstws}(2)) for $i\in \z$. Then there exist unique morphisms  $j_i:w_{\le i}N\to w_{\le i+1}N$ (for $i\in \z$) that make the corresponding triangles $w_{\le i}N\to w_{\le i+1}N\to N$ commutative. Moreover, the objects $\tilde N^{-1-i}=\co(j_i)[-1-i]$ belong to $\cu_{w=0}$ and there exists a complex $\tilde{t}(N)$  whose terms are $\tilde N^i$ (set in the corresponding degrees) and $\tilde{t}(N)\cong \tst(N)$  (in $K(\hw)$).

%Moreover, 
Furthermore, if %there exist 
$l\le m\in \z$ and % such that  
$w_{\le l}N=0$ %and $w_{\le m}N=N$ %and further?? hence bounded?? implies {iwcons}??
 then $w_{\le m}N$ belongs to the extension-closure of the set $\{\tilde N^j[-j],\ -m\le j\le -l\}$.

%????Moreover, the obvious modification of this statement  corresponding to contravariant weight-exact functors (cf. Proposition \ref{pbw}(\ref{idual})) is valid as well.

%\item\label{idualot} Assume that $\cu$ is a  stable symmetric monoidal $\infty$-category, and the tensor product restricts to $\hw$. Then the  functor $t^{st}$ is monoidal (with respect to the obvious tensor product on $K^b(\hw)$). 
%Moreover, if $\cu$ is rigid then the duality functor $\widehat$ sends $\hw$ into itself and 

\item\label{iwcons} If $M\in \cu_{w\le n}$ (resp. $M\in \cu_{w\ge n}$) then $t^{st}(M)$ belongs to $K(\hw)_{\wstu\le n}$ (resp. to $K(\hw)_{\wstu\ge n}$).

\item\label{iwcpu} Assume that $\ca$ is an additive covariant %(resp., contravariant)
  functor from $\hw$ into an abelian category $\au$. Then the functor $H^{\ca}$ %(resp. $H_{\ca}$)
	 that sends $M\in \obj \cu$ into the zeroth homology of the complex $\ca(M^i)$ %(resp. $\ca(M^{-i})$)  is (co)
	 is homological. % it respects coproducts if $\ca$ does and $\au$ is an AB4 abelian category.

Moreover, if $\ca=\ca'\circ \hf$ for some additive functor $\ca':\hw'\to \au$ in the setting of assertion \ref{iwcfunct} then  $H^{\ca}=H^{\ca'}\circ F$.

\item\label{iwcpuc} If $\au$ is an  AB4 abelian category then the functor $H^{\ca}$ as above is uniquely characterized by the following assumptions: it is homological, respects coproducts,  its restriction to the corresponding category $\bu$ (see Proposition \ref{pexw}) equals that of $\ca$, and its restrictions to $\bu[i]$ for $i\neq 0$ vanish.

%?? .+ cc???!! Moreover, this is the only  (co)homological functor (up to a unique isomorphism) whose restriction to $\hw$  equals $\ca$ and whose restrictions to $\hw[i]$ for $i\neq 0$ vanish. %this is the only (co)homological functor

%\item\label{iwcdet} Assume that $M$ is a $w$-bounded below object of $\cu$, $n\in \z$, and for any object $B$ of $\bu$, $i<n$, %$B\in \obj \bu$.
%and $\ca_B=\hw(B,-):\hw\to \ab$ we have $H^{\ca_B}(M[-i])=0$. Then $M\in \cu_{w\ge n}$.
\end{enumerate}
\end{pr}
\begin{proof}
%All these statements expect the % two last ones %
%last one 
Assertions \ref{iwcbase} and \ref{iwcfunct} easily  follow from  %Corollary 3.5
Remark 3.6  of \cite{sosnwc}, and the first part of assertion \ref{iwcpu} is obvious. %; cf. also Remark \ref{rwc}(2) below. % (along with its proof which is essentially self-dual); %По сути тот Йонеда, который используется -- это универсальное вложение в категорию с конечными пределами и сдвигами. Это то же самое, что универсальное вложение в категорию с конечными копределами и сдвигами из-за стабильности всего происходящего. Поэтому это всегда коммутирует с {op}.(cf. Remark 1.3.5(3)  of  \cite{bwcp});
  %see also  \S6.3 of \cite{bws} for the case where $\cu$ possesses a differential graded enhancement (and that is sufficient for our purposes below). %the definition and properties of $t$ described in \S1.3 of 
  %are stated in \cite{bwcp}; see Proposition 1.3.4(6,7,8,12) %, Remark 1.3.5(5) of ibid., 
	% and Theorem 2.1.2(1) of ibid.
	
	Next, we recall that the functor $\tst$ is "compatible" with the {\it weak} weight complex functor as defined in %loc. cit.;
	 \cite{bwcomp}; see Remark \ref{rwc}(2) below for more detail. Hence one can apply Proposition 1.3.4(4,6) and Lemma 1.3.2(3) of %\cite{bwcp} 
	 ibid. to obtain assertion \ref{iwcw}. 
	
	Similarly, 	%The 
	 (the easy) assertion \ref{iwcons} is given by Proposition 1.3.4(10) of ibid., % \cite{bwcp}. %Note however that to apply any of the results of ibid. 	that we mention in this proof 	one should recall that 
		%The first part of  Similarly, 
		 and the non-trivial (second) part of  assertion \ref{iwcpu} follows from Proposition 1.3.4(12) of loc. cit.  (see also Theorem 2.1.2 of ibid.). %whereas 	 Next,
		 Moreover, assertion \ref{iwcpuc} follows from Proposition 2.3.2(6) of ibid.; %note
	 to obtain this implication one should note  that the functors $\hw\to \au$ that  respect small coproducts are essentially in one-to-one correspondence with additive functors $\bu\to \au$ (since $\hw=\bu^{\widehat\oplus}$).   %; see also 	Remark \ref{rwc}(2) once again. 
	%assertion \ref{iwcpu} follows from 
		%Proposition 2.3.2(6) and  %is trivial, whereas the second one 
	%can be deduced from Remark 3.6  of \cite{sosnwc}; see also 	Remark \ref{rwc}(2) once again. 	%Assertion \ref{iwcpuc}
		 %Lastly, assertion \ref{iwcdet} easily follows from Proposition 2.2.3(1,2) of ibid.\footnote{Actually, to apply loc. cit. directly one should assumed that the category $\bu$ is essentially small. However, one can easily apply the argument used in the proof of loc. cit. to prove the general case of our assertion; note also that the smallness of the corresponding category $\bu$ is fulfilled in all the settings that we will consider below.}
	\end{proof}

\begin{rema}\label{rwc}
1. The term "weight complex" originates from \cite{gs}; yet the domains of the  ("concrete") weight complex functors considered in that paper were not triangulated. %, whereas the target was ("the ordinary") $K^b(\chowe)$.

2. In Proposition 1.3.4 of \cite{bwcomp} a certain (canonical) weak weight complex functor was defined as a functor from a category canonically equivalent to $\cu$ into the "weak" category $\kw(\hw)$. Now, our proof above depends on two observations.

Firstly, there exists a canonical additive functor $K(\hw)\to %K_{\w}
\kw(\hw)$, and the weak weight complex functor essentially factors through it; see (Remark 1.3.5(3) of %ibid. 
 \cite{sosnwc} along with Remark 3.6 of  \cite{sosnwc}).

Secondly,  the functors of the type $H^{\ca}$ as in part \ref{iwcpu} of our proposition (that were called {\it $w$-pure} ones in \cite{bwcomp}; %see Theorem 2.1.2 of ibid. (whereas 
the terminology was justified in Remark 2.1.3(3) of ibid.) factor through the weak weight complex functor; see Theorem 2.1.2 of  \cite{bwcomp}. %gives certain compatibility if combined with the {iwcpuc}??!
Moreover, the %existence 
 properties  of pure functors essentially do not depend on any enhancements (as mentioned in our proposition). In particular, it is easily seen that no $\infty$-lifts for $F$ are necessary for the second part of Proposition \ref{pwt}(\ref{iwcpu}). 

On the other hand, one can probably re-prove some of the statements above via arguments similar to that in Remark 3.6 of  \cite{sosnwc}.
\end{rema}

Next we pass to weight spectral sequences. We assume that $\cu$ possesses an $\infty$-enhancement since we want to cite the previous proposition.

\begin{pr}\label{pwss}
Adopt the assumptions of  Proposition \ref{pwt}; assume that $H$ is a homological functor $\cu\to \au$. 

 Then for any $M\in \obj \cu$  there exists a spectral sequence $T=T_w(H,M)$ with  $E_2^{pq}(T)=H^{G_{-q}}_{-p}(M)$, where %we write  $H_{-q}$ for $H\circ [q]$ %$E_1^{pq}(T)=H_{-q}(M^{p})$ %(see the convention for homolog (recall that $H_{-q}=H\circ [q]$), where  
 $G_{-q}$ is the restriction of the functor  $H_{-q}=H\circ [q]$ to $\hw$ (and respectively, $H^{G_{-q}}_{-p}=H^{G_{-q}}\circ [p]$; see also Proposition \ref{pwt}(\ref{iwcpu}). %$M^i$ and the boundary morphisms of $E_1(T)$ come from any choice of $t(M)$.
 
Moreover, $T_w(H,M)$ is $\cu$-functorial  in $M$ and in $H$ (with respect to composition of $H$ with exact functors of abelian categories), %;   starting from $E_2$; hence  $E_2^{pq}(T)=H^{G_{-q}}_{-p}(M)$, where $G_{-q}$ is the restriction of  $H_{-q}$ to $\hw$. Moreover, 
 and $T_w(H,M)$ %It
  converges to $H_{-p-q}(M)$ whenever $M$ is bounded below and $H$ kills $\cu_{w\ge i}$ %and $\cu_{w\le -i}$
  for $i$ large enough.
%The step of the filtration given by ($E_{\infty}^{l,m-l}:$ $l\ge n$)  on $H_{m}(M)$ (for some $n,m\in \z$) equals  $(W_{-n}H_{m})(M)$.
\end{pr}
\begin{proof}
%These statements  are  essentially  contained in Theorems 2.3.2 and  2.4.2  of \cite{bws}, respectively. 
This is (essentially) an easy combination of  Theorems 2.3.2   of \cite{bws} with our Proposition \ref{pwt}(\ref{iwcpu});  see also Remark \ref{rwc}(2). 
\end{proof}

%Below we will need the following simple statement that can be proved using weight spectral sequences.

\subsection{Chow weight structures on our categories}\label{swchow} % $R$-linear motivic categories}

%Now using the aforementioned statements we %results of  \cite{bwcomp} we 
 Using the results above we construct and study the Chow weight structures on $\dmerb$ and  on $\dmrr$.

\begin{pr}\label{pwchow}
Assume $r\ge -1$ and $M\in \obj \dmer$.
\begin{enumerate}
\item\label{ipinf} Then the categories $\dmer$ and $\dmr^r$ possess $\infty$-enhancements.

\item\label{ip1} There exists a left non-degenerate weight structure %smashing 
 $\wchow$ on $\dmerb$ that is  purely compactly generated by $\chower$ in the sense of Proposition \ref{pexw}; thus $\dmer_{\wchow\le 0}$ (resp.  $\dmer_{\wchow\ge 0}$) is the smallest subclass of $\obj \cu$ that is closed with respect to coproducts, extensions, and contains $\obj \chower[i]$ for $i\le 0$ (resp. for $i\ge 0$).

Respectively, %whose heart equals the coproductive hull of
 $\hwchow=\chower{}^{\widehat\oplus}$. %(see %\ref{snotata} Definition \ref{desmash}(\ref{idhull})  for this definition). 

%????Moreover, $DM^{eff}_{R, w\ge 0}=(\cup_{i< 0}\obj \chower[i])^{\perp}$ in $\dmerb$.

\item\label{iprest} % $\wchow$ restricts (see Definition \ref{dwso}(\ref{idrest}) to the subcategory $\dmer\lan r\ra$ of $\dmer$; the heart of this restriction equals $\chower{}^{\widehat\oplus}\lan r\ra$.
The functor $-\lan r+1\ra:\dmer\to \dmer$ is weight-exact with respect to $\wchow$.

Moreover, this functor is "strictly weight-exact", i.e., if  %$M\in \obj \dmer$ and
  $M \lan r+1\ra $ belongs to $\dmer_{\wchow\le 0}$ (resp. to $\dmer_{\wchow\ge 0}$) then $M\in \dmer_{\wchow\le 0}$ (resp. $M\in \dmer_{\wchow\ge 0}$) as well. 

\item\label{iptrun} Assume that for some $i\in \z$  there exist  choices of $\wchow{}_{\leq -i}M$ and $\wchow{}_{\leq -i-1}M$ that belong to $\obj \dmerb \lan r+1 \ra$. Then the corresponding object $\tilde{M}^{i}=\co(j_{-i-1})$ as mentioned in Proposition \ref{pwt}(\ref{iwcw}) belongs to $\dmer{}_{\wchow=0}\lan r+1 \ra$.

\item\label{iploc} %If $r\ge 0$ then the %
The   localization of $\dmerb$ by its subcategory $\dmerb\lan r+1\ra$ satisfies the conditions of Proposition \ref{winloc} with $\hu=\chower\lan r+1\ra$. Consequently, there exists a purely compactly generated weight structure $\wchow^{r}$ on $\dmr^r$  such that the localization functor $l^r:\dmer\to \dmr^r$ (see Definition \ref{dchowm}(4)) is weight-exact; moreover, $l^r$ respects coproducts and the compactness of objects.

Moreover, if  $-1\le s\le r$ then the obvious localization functor $l^s_r:\dmr^r\to \dmr^s$ is weight-exact and respects the compactness and coproducts as well.

\item\label{ipext} If $K/k$ is a field extension then the base field change functor $-_K:\dmer\to \dmer(K\perf)$ (see Definition \ref{dchowm}(2) and Proposition \ref{pmot} below) is  weight-exact and respects coproducts. 

\item\label{ip2} If $R$ is not torsion and $k$ is of infinite transcendence degree over its prime subfield, then the %corresponding
  weight structure $\wchow$ is right degenerate.
%\item\label{ip3} $\wchow$ restricts to a bounded weight structure on the subcategory of $\dmger$ and the heart of this restriction is equal to $\chower$.
%\item\label{ip4} Consider the subcategory $DM^{eff,\wchow}_{R}\subset \dmerb$, consisting of motives whose weight complexes are complexes of Chow motives (i.e., not of their coproducts).This is a triangulated subcategory, and $\wchow$ on $\dmerb$ restrict 
%to a weight structure on $DM^{eff,\wchow}_{R}.$ Moreover, the heart of this restriction is $\chower$.
\end{enumerate}
\end{pr}

\begin{proof} %All of these statements were proved in
Assertion \ref{ipinf} is obvious. %ask Sosnilo for a reference???

%Recall 
Now,  the subcategory $\chower$ compactly generates $\dmer$. Consequently, to obtain assertion \ref{ip1} it suffices to recall that $\chower$ is connective in $\dmer$ (see Corollary 6.7.3 of \cite{bev}) and apply Proposition \ref{pexw}. 

Next, %to prove assertion \ref{iprest} 
 we recall that  the functor $-\lan r+1\ra$ respects coproducts and sends $\chower$ into itself.
 %$\chower\lan r\ra\subset \chower$
  Applying the explicit description of $\wchow$ we obtain that $-\lan r+1\ra$ is weight-exact. Moreover, if $M \lan r+1\ra\in \dmer_{\wchow\le 0}$ (resp.  $M \lan r+1\ra \in \dmer_{\wchow\ge 0}$) then applying this weight-exactness we obtain that $M \lan r+1\ra\perp \dmer_{\wchow\ge 1}\lan r+1\ra$ (resp. $ \dmer_{\wchow\le -1}\lan r+1\ra \perp M$). Combining Proposition \ref{pbw}(\ref{iort}) with the Cancellation Theorem (which %that %i.e., with the fact 
	says that $-\lan r+1\ra$ is fully faithful) we conclude the proof.

\ref{iptrun}.	Proposition \ref{pwt}(\ref{iwcw}) says that $\tilde{M}^{i}$ belongs to  $\dmer_{\wchow= 0}$. Since $\tilde{M}^{i}$ also belongs to $\obj \dmer \lan r+1\ra$, the previous assertion implies that $\tilde{M}^{i}=N\lan r+1\ra$, where $N$ belongs to   $\dmer_{\wchow\le 0}\cap  \dmer_{\wchow\ge 0}= \dmer_{\wchow= 0}$.

%Consequently, if $M \lan r+1\ra \in \dmer_{\wchow= 0}$ then $M\in \dmer_{\wchow= 0}$.
	% obtain the "moreover" part  of  assertion \ref{iprest}. Its last part follows from the previous statement immediately.

%Next, a
%Applying this observation 
 Since $\chower\lan r\ra\subset \chower$, we also obtain that the first part of assertion \ref{iploc} follows from Proposition \ref{winloc}. 
% since $\chower\lan r\ra\subset \chower$.
Next, to obtain the "moreover" part of the assertion one can apply Proposition \ref{winloc} for $\bu$ and $\hu$ consisting of  $\mgr^r(\spv)$ and $l^r(\mgr(\spv)\lan s+1\ra)$, respectively.
%images in $\dmrr$ (via $l^r$) of $\ob
 
\ref{ipext}. We recall that if $f: \spe K\perf\to \spe k$  is the corresponding morphism then $-_K$ %is essentially 
 can be defined as the restriction to $\dmer$ of the functor $f^*:\dmr\to \dmrbl$; see Proposition \ref{pmot}(\ref{ie1}) below. Now, the latter functor  respects coproducts % essentially by definition; see  Remark 1.2 and %, Definition 1.5, and Proposition 8.1(c) 
%Proposition 4.3 and Theorem 5.1
%Theorem 3.1 of \cite{cdint}.   %for the corresponding morphism $f:\spe K\perf\to \spe k$ 
 since there exists a functor $f_*$ that is right adjoint to it; % $f^*=-_K$; 
see %Remark 1.2 and  
 Theorem 3.1 of \cite{cdint}, and Definitions  1.1.12 and 1.4.2 of \cite{cd}.  Hence $-_K$ respects coproducts as well.

Next, $f^*$ sends (effective) Chow motives over $k$ into that over $K\perf$ %immediately by
 (see Proposition \ref{pmot}(\ref{ie1}) below), applying the explicit description of $\wchow$ once again we obtain that  $-_K$ 
 %$f^*$
  is weight-exact indeed. 

%See Proposition 3.2.6 and 
Lastly,  assertion \ref{ip2} is given by Proposition 3.2.6  of \cite{bwcomp}. 
\end{proof}

%A corollary????! 

%\subsection{On %certain properties of  higher Chow groups and motivic homology}\label{phcgmh}

\section{On Chow-weight homology and its relation to motivic homology}\label{scwh}

In \ref{sscwh} we define and study Chow-weight homology functors $\dmer\to \ab$; these statements generalize the ones of \cite[\S3.1]{bsoscwhn}.
%extend the properties of Chow-weight homology estab

In \ref{scwhp} we extend the equivalent criteria for the vanishing of Chow-weight homology from $\dmger$ (as studied in ibid.) to $\dmerpl$.

In \ref{smoth} we prove our central Theorem \ref{main} that relates the vanishing of motivic homology in a staircase range to certain Chow-weight homology vanishing conditions. This is related to the study of {\it slices}.

\subsection{On Chow-weight homology of general motives: basic properties}\label{sscwh}

%see Definition \ref{deffdim}(6) for the notation that we use here.????!!!!

\begin{defi}\label{dcwh}
%1.  We will write $\wcr(M)$ for a choice of a weight complex for a   motive $M\in \obj \dmerb$  with respect to the Chow weight structure for $\dmerb$ (it is a $\chower{}^{\widehat\oplus}$-complex,  and one can assume that $\wcr$ is an exact functor $\dmerb\to K(\chower$$^{\widehat\oplus})$. 
%2. 

1. Let $i, l,j\in \z$; % and %$j\ge 0$; 
 let $K$ be a field extension of $k$.

Then we will write  $\chw_j^{i}(-_K,R,l)$ for the functor $H^{\ca}\circ [i]$, %  (see \S\ref{snotata} for a discussion on this convention), 
 where  we apply Proposition \ref{pwt}(\ref{iwcpu}) to the weight structure $\wchow$ (see Proposition \ref{pwchow}(\ref{ip1}) and $\ca$ is the restriction of the functor $N\mapsto \chowm_j(N_K,R,l)$ (see Definition \ref{dchowm}(3)) to $\hwchow$. % and })). 
%For an object $M$ of $ \dmerb$ and $(M^s)$ being a choice of $\wcr(M)$, we define the abelian group $\chw_j^{i}(M_K,R,l)$ as the $i$-th homology of the complex $\chowm_j(M^s_K,R,l)$;
 We will often omit $R$ in this notation when its choice is clear. %????, and omit $K$ if it equals $k$.

 Moreover, %at some occurencies 
 we will often write $\chw_j^{i}(M_K,R)$ for  $\chw_j^{i}(M_K,R,0)$. 

2. We will use the notation $\dmrwp$ for the class of   $\wchow-$bounded below motives (see Definition \ref{dwso}(\ref{lrbo})). 

3. We will %also need the following notation: 
 need the following %extension of Definition \ref {dchowm}(1)
 convention (cf. \S\ref{smot}): $\chower\lan {+\infty}\ra=\dmer\lan {+\infty}\ra=\ns$, $l^{+\infty}=l^{+\infty}_{+\infty}$ is the identity on $\dmer$, $l^j_{+\infty}=l^j$,  $\wchow^{+\infty}=\wchow$, and $+\infty+1=+\infty$.

4. We will write $\thomr$ for the $R$-linear version of the homotopy $t$-structure of Voevodsky; see  %\S3.1.2 
 \S4.4 of \cite{bev} or Example 2.3.13 of \cite{bondegl}.  %or  Corollary 5.2 of  \cite{degmod}
 % for its %$R$-linear  %$\dmer$-version. 
\end{defi}

Let us prove some properties of these functors.

\begin{pr}\label{pr1}

Let $i,l,j, K$ be as above %, $j\ge 0$,
  and $r\in [0,+\infty]$.

1. $\chw^{i}_{j}(-_K,R,l)$ %yields
 is  a  homological functor on $\dmerb$ that  respects coproducts. % (and does not depend on any choices). 
 Moreover, this functor factors through the base field change functor  $\dmerb \to \dmerb(K^{perf})$.

2. Assume   $r\ge j+l$. Then  the functor $\chw_j^{0}(-_K,R,l)$ kills $\dmerb\lan r+1\ra$; thus it induces a well-defined functor %$\dmerb/\dmerb\lan r+1\ra
$\dmr^r\to \ab$ (see Definition \ref{dchowm}(4)). 

Moreover, this functor is pure with respect to the  weight structure $\wchow^r$ (see Proposition \ref{pwchow}(\ref{iploc})).%corresponding???? 

3. %compute the composition functor???
%Suppose $N\in\dmer/\dmer\lan j+1\ra _{\wchow\ge 0}$.Then 
For any smooth projective connected variety $P/k$ the functors %$\dmer/\dmer\lan j+1\ra 
$\dm^{j}_{R} (l^j(\mgr(P)\lan j\ra),-)$ and $ \chw_{j}^0(-_{k(P)},R)$ are canonically isomorphic;  note that the latter functor is well-defined according to the previous assertion.

 %4. Assume $N\in \dmer/\dmer\lan r+1\ra_{\wchow\ge -n}$ for some $n\in \z$. Then $\chw_{j}^{i}(N_K,l)=\ns$ for all $i>n$, $j\le r-l$ (note that these Chow-weight homology groups of $N$ are well-defined). 

4. Let $N$ belong to %$DM^{eff}_{R,\wchow\ge -n}$ (resp. to 
 $\dmrr_{\wchow^r\ge -n}$.
Then $\chw_{j}^{i}(N_K,l)=\ns$ whenever either $i>n$ %(resp. for $i>n$
and $j\le r-l$ or if $l<0$.

5. %Moreover, if  
Moreover, %If  %$0\le m\le r$
 if $m\in [0,r]$ then the following %additional 
 assumptions on $N\in \dmrr_{\wchow^r\ge -n}$ %(as in the previous assertion) 
 are equivalent.
%and   we assume

(a). $\chw^{i}_{j}(N_K)=\ns$ for all $0\le j\le m$ and all function fields $K/k$;

(b). The %corresponding image $N_m$ of $N$ in $\dmr^m$
 object $N_m=l^m_r(N)$   belongs to $\dmr^m{}_{\wchow^m\ge 1-n}$.

(c).  There exists a choice of  $\wchow^r{}_{ \le -n}N$ that belongs to %$\obj \dmr\lan m+1\ra$ (resp. to 
$l^r(\obj \dmr\lan m+1\ra)$.  
%For any  choice of a $-n$-weight decomposition  $\wchow_{ \le -n}N\stackrel{g}{\to} N\to \wchow_{\ge 1-n}N$ (resp. $\wchow^r_{ \le -n}N\stackrel{g}{\to} N\to \wchow^r_{\ge 1-n}N$)  of $N$  the morphism $g[n]$ can be factored through a coproduct of objects of $\chower\lan m+1\ra$ (resp. through an image of a coproduct of this form in  $\dmr^r$). %$\dmerp/\dmerp\lan r+1\ra$).

6. %Let $\thomr$ be the $R$-linear version of the homotopy $t$-structure of Voevodsky; see  \S4.4 of \cite{bev} %or  Corollary 5.2 of  \cite{degmod}
  %for its %$R$-linear 
 %$\dmer$-version. 
The class  $DM^{eff,\thomr\le 0}_{R}$ (see Definition \ref{dcwh}(4))  equals the smallest extension-closed subclass of $\obj\dmerb$ that is closed with respect to coproducts and contains $\obj\chower(a)[a+b]$ for all $a,b \ge 0$. Moreover, % and 
 the functor $\chw_{j}^{i}(N_K,l)=\ns$ kills this class whenever $i>j+l$.
%If $N\in  DM^{eff,\thomr\le 0}_{R}$  and $i>j+l$ then  $\chw_{j}^{i}(N_K,l)=\ns$. %; here 
%7.  For any $N\in DM^{eff}_{R,\wchow\ge 0}$ and any field extension $K/k$ we have  $\chw_{j}^{0}(N_K,R)\cong \chowm_{j}(N_K,R)$. 
\end{pr}
\begin{proof}
1. All the statements easily follow from Proposition \ref{pwt}(\ref{iwcpu}, \ref{iwcpuc}) along with Proposition \ref{pwchow}(\ref{ipext}); note here that %To see this one should only note that 
the corresponding motivic homology functors %along with the base field change functor $\dmer\to \dmerb(K^{perf})$ %along with  
  %are  well-known to 
	respect coproducts  since they are corepresented by compact objects of $\dmrbl$ (see \ref{smot}). %, and the latter functor is  weight-exact immediately from Proposition \ref{pmot}(\ref{ie1}) below. % and the motivic homology functors respect coproducts in $\chower$ since they do so in $\dmrb$.
%The first part of the assertion is just a particular case of %Proposition \ref{pbwcomp}(\ref{iwcoh}); recall here that $\wchow$ is smashing by Proposition \ref{pwchow}(\ref{ip1}). The second part follows immediately from the weight-exactness of this base field change functor.

2. To prove the first part of the assertion we should verify that $\chw_j^{0}(-_K,R,l)\circ \lan r+1 \ra=0$. Recall that the functor $-\lan r+1\ra:\dmer\to \dmer$ is weight-exact with respect to $\wchow$ by Proposition \ref{pwchow}(\ref{iprest}); hence  Proposition \ref{pwt}(\ref{iwcpu} , \ref{iwcpuc}) reduces the statement to the vanishing of the restriction of  $\chowm_j(M^s_K,R,l)$ to $\chower\lan r+1\ra$. The latter fact is given by Proposition \ref{pvan}.

To prove the "moreover" statement we invoke Proposition \ref{pwt}(\ref{iwcpuc}) once again. %Clearly we can assume $i=0$.
  Since the functor $\dmerb\to\ab$ induced by $\chw_j^{0}(-_K,R,l)$ respects coproducts, % (as well; one may apply the statements mentioned above once again), 
	 it suffices to note that its restrictions to $\mgr^r(\spv)[s]$ for $s\neq 0$ vanish since the functor $\chw_j^{0}(-_K,R,l)$ is pure with respect to $\wchow$. %???Proposition \ref{pwchow}(\ref{ipext}); here we apply % the aforementioned properties of base change and motivic homology.

%The statement  easily follows from  Proposition \ref{pvan}(2) implies that the corresponding functor $\ca$ kills 
% combined with the "functoriality" with respect to weight-exact functors provided by Proposition \ref{pwt}(\ref{iwcpu}) (see also Proposition \ref{pwchow}(\ref{iploc})).

3. %Apply uniqueness or compute directly???
Since the object $l^j(\mgr(P)\lan j\ra)$ is compact in $\dmr^j$ (see Proposition \ref{winloc}(1)), both of the functors in question respect coproducts.  Now, they are also homological, and the functor  $ \chw_{j}^0(-_{k(P)},R)$ is $\wchow$-pure by definition; % (see Remark \ref{rwc}(2));????
  thus to obtain the %equality
	  isomorphism in question it suffices to compare their restrictions to the categories $l^j(\chower)[i]\subset \obj \dmr^j$ for $i\in \z$. 

Next, recall that the localization $\dmger/\dmger\lan j+1\ra$ embeds into $\dmr^j$ by Proposition \ref{winloc}(1). Thus %%we can apply 
Proposition 2.2.5(6) of \cite{bsoscwhn}  yields the result easily. %to obtain the result.

4. Clearly, we can assume %that 
 $n=0$. 

By Proposition \ref{pwt}(\ref{iwcons}) the corresponding weight complex $\tst(N)$ is homotopy equivalent to a complex  concentrated in non-positive degrees. % at most $0$. 
 Next we recall that the functor $\chw_{j}^{0}$ %(N_K,l)=\ns$ 
 is pure with respect to %the Chow weight structure on the corresponding category
 $\wchow^r$ (see Definition \ref{dcwh}(1) and assertion 2 of this proposition); applying the relation of pure functors to weight complexes we obtain that $\chw_{j}^{i}(N_K,l)=\ns$ for all $i>0$ (and the corresponding values of $j$) indeed.

Lastly,  $\chw_{j}^{i}(N_K,l)=\ns$ if $l<0$ since the corresponding restriction $\ca$ of the functor $N\mapsto \chowm_j(N_K,R,l)$ to $\hwchow$ is zero immediately from the connectivity of the category $\chower(K\perf)$. %all Chow motives?? Cancellation??

5. %Now let us prove the second part of the assertion. 
%Since 
If $j<m$ then the functor $\chw_{j}^{i}(-_K)=\ns$ factors through $l_m^r$ by assertion 2; hence the implication (b) $\implies$ (a) %is given by 
 follows from assertion 4. Next, condition (b) trivially implies condition (c) if $r=+\infty$, and it %also implies condition (c) 
does so in the case %$r\in \z$ 
 $r<+\infty$ as well by Theorem 3.3.1 of \cite{bsosnl}; cf. Remark 3.3.2(1) of loc. cit. %(see also ; cf. (the proof of) \cite[Proposition 1.3.2(2)]{bscwh}.
  Moreover, recall that %both the functor $l^m$ and 
	 the functor $l^m_r$ is weight-exact by Proposition \ref{pwchow}(\ref{iploc}); hence condition (c) implies condition (b).
		
		%condition (c)  implies that $N_{m}$ is a retract of $\wchow_{ \le -n}N$ 

%Next, 
It remains to verify that (a) implies (b). We assume $n=0$ once again. %If condition (b) is fulfilled then 

First assume that $m<+\infty$. %Then
  It clearly suffices to verify that condition (a) implies the following: % (under our vanishing assumptions): 
 if  $-1\le s<m$ %(resp. $r\le s<m$)
  and %the  image $N_s$ of $N$ in  $\dmr^s$ belongs to
	$l_r^s(N)\in\dmr^s{}_{\wchow^s\ge 1}$ then %the corresponding image
	 $N_{s+1}$ belongs to $\dmr^{s+1}{}_{\wchow^{s+1}\ge 1}$.

We take a weight  decomposition  % $\wchow_{ \le 0}N\stackrel{g}{\to} N\to \wchow_{\ge 1}N$ (resp.  
%\begin{equation}\label{wdsp}
$\wchow^{s+1}{}_{ \le 0}N_{s+1}\stackrel{g_{s+1}}{\to} N_{s+1}\to \wchow^{s+1}{}_{ \ge 1}N_{s+1}$ %\end{wdsp}
 of $N_{s+1}$, and 
%We apply the corresponding 
 apply the localization $l^s_{s+1}:\dmr^{s+1}\to \dmr^{s}$. Since %it 
  $l^s_{s+1}$ is weight-exact, % (see Proposition \ref{pwchow}(\ref{iploc})), 
 we have $l^s_{s+1}(\wchow^{s+1}{}_{ \le 0}N_{s+1})\in \dmr^{s}{}_{\wchow^{s}\le 0}$. Since $l_r^s(N)=l^s_{s+1}(N_{s+1})$, the orthogonality axiom (iii) of Definition \ref{dwstr} gives $l^s_{s+1}(g_{s+1})=0$. 

Next, Proposition \ref{pbw}(\ref{iwd0}) implies $\wchow^{s+1}{}_{ \le 0}N_{s+1}\in   \dmr^{s+1}{}_{\wchow^{s+1}=0}$. Thus combining  %Corollary \ref{cfactor}(2) 
 Proposition \ref{pbw}(\ref{ifactp}) with Proposition \ref{pwchow}(\ref{iprest})  we obtain that $g_{s+1}$ factors through an element of $ \dmr^{s+1}{}_{\wchow^{s+1}=0}\lan s+1 \ra$; thus it factors through a coproduct of $\mgr^{s+1}(P_a)\lan s+1\ra$ for some (connected) varieties $P_a\in\spv$. 

Applying assertion 3, we obtain that  $\mgr^{s+1}\lan s+1\ra(P_a)\perp N_{s+1}$ %We note that  $\mgr^{s+1}\lan s+1\ra(P_a)\perp N_{s+1}$ by assertion 3.
 since  $ \chw_{s+1}^0(N_{s+1,k(P_a)})=0$.   Thus $g_{s+1}=0$. It follows that $N_{s+1}$ is a retract of $\wchow^{s+1}{}_{ \ge 1}N_{s+1}$; hence $N_{s+1}$ belongs to $\dmr^{s+1}{}_{\wchow^{s+1}\ge 1}$ indeed.

It remains to consider the case $m=r=+\infty$. Similarly to the argument above, it suffices to verify that $\mgr(P_a)\perp N$ for any smooth projective $k$-variety $P_a$. Now, assume that $P_a$ is of dimension  $d$ and apply the equivalence of our conditions in the case $m=d$ (note that we have just proved it). Condition (b) in this case gives a weight decomposition triangle with $\wchow^{d+1}{}_{ \leq 0 }N_{d+1}\in \obj(\chower{}^{\widehat\oplus})\lan d+1\ra$ (see the arguments above). The definition of Chow groups immediately implies that $\mgr(P_a)\perp\mgr(\spv)\lan d+1\ra$; hence $\mgr(P_a)\perp \wchow^{d+1}{}_{\leq 0}N_{d+1}$ as well. Since $\mgr(P_a)\perp \wchow^{d+1}{}_{ \ge 1}N_{d+1}$ by the orthogonality axiom for $\wchow$, we obtain that $\mgr(P_a)\perp N$ indeed.

6. %Recall that $DM^{eff,\thomr\le 0}_{R}$ is the smallest extension-closed subclass of $\obj\dmerb$ that is closed with respect to coproducts and contains $\obj\chower(a)[a+b]$ for all $a,b \ge 0$ (see 
 The first part of the assertion is given by Theorem 2.4.3 along with Example 2.3.13 %???5(1)
  of \cite{bondegl}; it also can be easily deduced from Theorem 2.2.1(3) of  \cite{bzp} (cf. also Theorem 6.2.1(1) of \cite{mymot}).
Next, the Chow-weight homology functors respect coproducts; hence 
 the vanishing in question follows from %this statement along 
 this first part combined with Proposition \ref{pvan}.
%7. Immediate from Corollary \ref{cfactor}(1).
\end{proof}

\begin{rema}\label{bbard}
Since the weight structure $\wchow$ is right degenerate (at least, %if $R$ is not torsion and $k$ is of infinite transcendence degree over its prime subfield; 
 in some cases; see Proposition \ref{pwchow}(\ref{ip2})), and for right weight-degenerate objects $M$ the weight complex $t_{R}(M)$ vanishes,  below we will mainly concentrate on $\wchow$-bounded below  motives. %objects. %As an example,
 Recall here that Lemma 2.4 of \cite{ayconj} gives an interesting example of right $\wchow-$degenerate motif; see Proposition 3.2.6 of \cite{bwcomp}. Note also that this motif is infinitely effective, i.e., belongs to $\cap_{r\ge 0}\obj\dmer\lan r\ra$; cf. Remark \ref{rinfe} below.

Another evidence for %the %fact that it is difficult to apply 
  certain problems with applying arguments %of the sort we use here 
	 similar to our ones to objects that are not $w$-bounded below is given by Remark 2.2.6(3) of ibid. 

%nafig??? In the introduction?! Да, наверное туда.
2. Recall that Chow-weight homology for geometric motives was introduced and thoroughfully studied in \cite{bsoscwhn}. Our version of the theory is the only pure extension of this homology theory that respects coproducts (see Proposition \ref{pwt}(\ref{iwcpuc})).
%??However, the definition of Chow-weight homology only on geometric motives does not allow one to prove some statements about the relation with motivic homology (see Lemma \ref{hhom} below).
\end{rema}

\subsection{Some Chow-weight homology vanishing criteria}\label{scwhp} %{On properties equivalent to the vanishing of Chow-weight homology}

To formulate our statements in the most general case we recall the following technical definition.

\begin{defi}\label{dreasi}
1. Let $\ii$ be a subset of $\z\times  [0,+\infty)$ (see \S\ref{snotata}).

We will call it a {\it staircase} set if for any $(i,j)\in \ii$ and  $(i',j')\in \z\times  [0,+\infty)$ such that $i' \ge i$ and $ j' \le j$ we have $(i',j') \in \ii$. 

For $i\in \z$  the minimum of $j\in [0,+\infty]$ such that $(i,j) \notin \ii$ will be denoted by $a_{\ii,i}$.

2. For $ m \in \mathbb{Z}$ we will write $d_{\leq m} \dmerb$ for the localizing subcategory of $\dmerb$ generated by $\{M_{R}(X)\}$ for $X$ running through smooth $k$-varieties of dimension at most $m$; thus this category is zero if $m<0$.

Moreover, we will denote by $d_{\le m}\chower$ the subcategory of $\chower$ consisting of the retracts of the motives of smooth projective varieties of dimension at most $m$.
\end{defi}

\begin{rema}\label{rstair}
Obviously, $\ii\subset \z\times  [0,+\infty)$ is a staircase set if and only if it equals the union of the strips $\bigcup\limits_{(i_0,j_0) \in \ii} \ii_{i_0,j_0}$, where $\ii_{(i_0,j_0)} = [i_0,+\infty) \times [0,j_0]$. 

Consequently, the union of any set of staircase sets is a staircase set as well; cf. Theorem \ref{tstairs}(\ref{itstairs5}) below.
\end{rema}

Now we study the vanishing of $\chw_{*}^{*}(M_K)$ in staircase degrees.

\begin{theo}\label{tstairs}
Let  $\ii\subset \z\times  [0,+\infty)$ be a staircase set, $M\in  \dmerpl$ (see Definition \ref{dwso}(\ref{lrbo})). 
Then the following statements are valid.

 %1. The vanishing of $\chw_{j}^{i}(M_K)$ for all function fields $K/k$ and all $(i,j)\in \ii$ is equivalent to the same vanishing for all field extensions $K/k$.
 \begin{enumerate}
\item\label{itstairs2}
  %Suppose that $\ii$  is a staircase set. Then the 
	The following conditions are equivalent.

A. $\chw_{j}^{i}(M_K)=\ns$ for all function fields $K/k$ and $(i,j) \in \ii$.

B. The object $l^{j}(M)$  belongs to $DM_{\wchow \ge -i+1}^{R,j}$ whenever $(i,j) \in \ii$.

C. For any $i\in \z$ there exists a choice of  $\wchow{}_{\le -i}M $ that belongs to  $\obj\dmerb\lan a_{\ii,i}\ra$. %??; recall here that we define $\chower \lan +\infty\ra = \dmerb\lan  +\infty\ra=\ns$.

D. $M$ belongs to the %extension-closure of $\cup_{i\in \z }(\obj (\chower)^{\widehat \oplus}[-i]\lan a_{\ii,i}\ra)$.
smallest extension-closed class $D_{\ii}$ of objects of $\dmer$ that is also closed with respect to coproducts and contains $\cup_i\obj (\chower)[-i]\lan a_{\ii,i}\ra$.

E. There exists %a choice of a weight complex for 
 a $\chower$-complex $\tilde{t}M\cong t(M)$ such that its $i$-th term $\tilde M^i$  is an object of  $\chower\lan a_{\ii,i}\ra^{\widehat\oplus}$. %is a coproduct of $j+1$-effective Chow motives whenever $(i,j) \in \ii$.

\item\label{itstairs3}
 If %$\ii$  being a staircase set and 
 $M\in DM^{eff}_{R,[a,b]}$ (for some $a\le b\in \z$) then %the (equivalent) conditions of the previous assertion are fulfilled
 $M$ belongs to $D_{\ii}$ (see Condition \ref{itstairs2}.D)    if and only if $M$ belongs to the extension-closure of $\cup_{-b\le i\le -a }(\obj (\chower\lan a_{\ii,i}\ra^{\widehat \oplus})[-i]$).

\item\label{itstairs4}
%Assume that 
 If $M$ is of dimension  at most $r\ge 0$ %and that $\ii$ is a staircase set. Then 
 then the (equivalent) conditions of  assertion  \ref{itstairs2} are also equivalent to the following ones:

A'. $\chw_{j}^{i}(M_K)=\ns$ %for all function fields $K/k$ and 
whenever  $(i,j) \in \ii$ and $K=k(P)$, where $X$ is a smooth projective $k$-variety of dimension at most $r-j$.

C'. For any $i\in \z$ there exists a choice of  $\wchow{}_{\le -i}M $ that belongs to  $\obj(d_{\le r-a_{\ii,i}}\dmerb)\lan a_{\ii,i}\ra$. %??; recall here that we define $\chower \lan +\infty\ra = \dmerb\lan  +\infty\ra=\ns$.

D'. $M$ belongs to the %extension-closure of $\cup_{i\in \z }(\obj (\chower)^{\widehat \oplus}[-i]\lan a_{\ii,i}\ra)$.
smallest extension-closed class %$D'_{\ii}$???
  of objects of $\dmer$ that is also closed with respect to coproducts and contains $\cup_i\obj (d_{\le r-a_{\ii,i}}\chower)[-i]\lan a_{\ii,i}\ra$.

E'. There exists %a choice of a weight complex for 
 a $\chower{}^{\widehat \oplus}$-complex $\tilde{t}M\cong t(M)$ such that its $i$-th term $\tilde M^i$ is an object of  $(d_{\le r-a_{\ii,i}}\chower)\lan a_{\ii,i}\ra^{\widehat\oplus}$.
%is a coproduct of $j+1$-effective Chow motives whenever $(i,j) \in \ii$.

  Moreover, a similar modification can also be made in assertion \ref{itstairs3}.

\item\label{itstairs5} Assume that $\ii_j$ are staircase sets for $j$ running through some index set $J$, and $\ii=\cup\ii_j$. Then 
$M$ belongs to $D_{\ii}$ if and only if it belongs to $\cap_j D_{\ii_j}$.
%for a fixed  an object $M$ of 

%the %(equivalent) conditions of assertion  \ref{itstairs2} are fulfilled for  all $\ii_j$  if and only if they are fulfilled for $\ii$. 
\end{enumerate}
\end{theo}

\begin{proof} 
%1. Easy from Proposition 2.3.4 of \cite{bsoscwhn}.
 \ref{itstairs2}. Consider the strip $\ii_{(i_0,j_0)} = [i_0,+\infty) \times [0,j_0]$. Applying  Proposition \ref{pr1}(5) %repeatedly 
 inductively we easily obtain that the vanishing of $\chw_{j}^{i}(M_K)$ for all $(i,j) \in \ii_{(i_0,j_0)}$ is equivalent to $l^{j_{0}}(M)\in DM^{R,j_{0}}_{\wchow \ge -i_{0}+1}$; cf. the proof of \cite[Theorem 3.2.1(2)]{bsoscwhn}.
Using Remark \ref{rstair} we get the equivalence A$\Leftrightarrow$ B. 
Similarly, the  %implication B$\Rightarrow$C %instantly follows from Proposition \ref{pr1}(5).\\  %\ref{pr1}(4)???
 %easily follows from the weight-exactness; see 
% can easily be verified using the arguments used in the proof of  Proposition \ref{pr1}(4).
equivalence B $\Leftrightarrow$ C  follows from Proposition \ref{pr1}(5) as well.

Next,  we can take $w_{\le i}M=0$ for $M$ small enough. Thus if condition C is fulfilled then combining Proposition \ref{pwchow}(\ref{iptrun}) with the extension-closure statement in Proposition \ref{pwt}(\ref{iwcw})  we obtain that the choices of $\wchow{}_{\le -i}M$ provided by condition C actually belong to the class $D_{\ii}$. 

Now, $\wchow$ is smashing and left non-degenerate by Proposition \ref{pwchow}(\ref{ip1}). Hence for (any object of $M$ of $\dmer$ and) any choices of  $\wchow{}_{\ge j}M$ there exists a ({\it countable homotopy colimit}) distinguished triangle $\coprod_{j\ge 0} \wchow{}_{\ge j}M\to \coprod_{j\ge 0} \wchow{}_{\ge j}M \to M\to \coprod_{j\ge 0} \wchow{}_{\ge j}M[1]$; see Theorem 4.1.3(1,2), Definition 4.1.1, and Remark 1.2.6(1) of \cite{bsnew}. %Now, if condition C is fulfilled then applying Proposition \ref{pexw} (along with Proposition \ref{pwchow}(\ref{iprest}))  %there exist choices of 
%we can %assume that all  take 
 %we obtain the existence (of choices) of 
 Thus if we take $ \wchow{}_{\ge j}M $ that belong to  $D_{\ii}$ then we conclude that $M$ belongs to $D_{\ii}$ as well. % the class specified in condition D. % by Proposition \ref{pbwcomp}(\ref{iwext}). %ok?? see assertion 3. %???
Consequently, condition C implies condition D. 

%Moreover, condition D implies condition E by 

Next, we prove the implication $D \Rightarrow E$ %can be proved 
 similarly to  Proposition 2.3.2(9) of \cite{bwcomp} (thus one can use weak weight complexes instead of strong ones in our proof). %as we do). %Since the
 Recall that the  functor $\tst$ is exact and respects coproducts. Since the class $C_{\ii}$ of objects of $K(\chower)$ that are isomorphic to the ones satisfying our effectivity assumptions on terms is obviously closed with respect to extensions and coproducts, the class $D'_{\ii}$ of those %$N\in \obj \chower$ 
 $N\in \obj\dmer$ such that $\tst(N)\in  C_{\ii}$ is closed with respect to extensions and coproducts as well. Since $D'_{\ii}$ obviously contains  $\cup_i\obj (\chower)[-i]\lan a_{\ii,i}\ra$, we obtain $D_{\ii}\subset D'_{\ii}$.
%the equivalence between conditions C, D and E follows from Remark \ref{trunprop} and Proposition \ref{pbwcomp}(\ref{iwext}).\\

 For the remaining implication $E \Rightarrow A$ note that if $\tilde{t}(M)=(\tilde M^i)$ then $\chw_{j}^{i}(M_K)$ is a subquotient of $\chowm_{j}(\tilde M_{K}^{i}, R)$, and the latter  group vanishes by Proposition \ref{pvan}.

 \ref{itstairs3}. Assume that $M$ satisfies Condition $C$ of the previous assertion and belongs to $ DM^{eff}_{R,[a,b]}$. Then we can set $\wchow{}_{\le a}M = 0$. Moreover,  condition D in condition \ref{itstairs2} implies that $M$ % belongs to
is an object of $ \dmer\lan a_{\ii,-b}\ra$; thus if we set 
 $\wchow{}_{\le b}M = M$ and choose $\wchow{}_{\le -i}M$ to belong to  $\obj \dmer\lan a_{\ii, i}\ra$ for $-b<i<-a$ then this condition would be fulfilled for our choices of weight truncations for $-b\le i \le -a$. 

Now we  combine Proposition \ref{pwchow}(\ref{iptrun}) with % the extension-closure statement in 
 Proposition \ref{pwt}(\ref{iwcw}). Similarly to the proof of the implication   $C \Rightarrow D$ in the previous assertion, we obtain that $\wchow{}_{\le b}M = M$ belongs the extension-closure in question.

\ref{itstairs4}. Obviously, condition A of assertion \ref{itstairs2} implies our condition A', whereas conditions C', D', and E' imply conditions \ref{itstairs2}.C, \ref{itstairs2}.D, and \ref{itstairs2}.E respectively. Thus it suffices to verify that these conditions are equivalent, and also imply the "bounded dimension" version of assertion \ref{itstairs3}.

%pre-requisites from \cite{binters}
Next, it can be easily checked that the arguments above carry over to our setting if one invokes the following %pre-requisites from \cite{binters}
 statements: for any $j\ge 0$ the weight structure $\wchow$ restricts to the (triangulated) subcategory of $d_{\le r}\dmer$ whose objects are the $j$-effective and $\wchow$-bounded below (in $\dmer\supset d_{\le r}\dmer$) motives, 
%\cap \dmer\lan j\ra$ for any $j\ge 0$, 
and the heart of this restriction equals $(d_{\leq r-j} \chower)\lan j\ra^{\widehat\oplus}$. Now, this statement easily follows from Theorem 2.2 of \cite{binters} (along with its proof and Proposition 1.7 of ibid. that give the calculation of the heart). % where the corresponding heart is essentially described).  %restrict to $\wchow$-bounded below objects or apply?!

We leave the detail for this argument to the reader, since we will not apply this assertion below. %One
 We only note that %should consider
 we propose to take the subcategories of dimension at most $r$ in the localizations of the type $\dmr^j$ in it and not to consider the corresponding localizations of  $d_{\le r}\dmer$ (even though the latter can probably be used as well; cf. Proposition 2.2.5(7) of \cite{bsoscwhn}). %allows to consider "bounded dimension" localizations as well??!, since it is not clear to the authors whether the corresponding calculations can be made in localizations of $d_{\le r}\dmer$.

%The argument is essentially the same as for the equivalence A$\Leftrightarrow$ B in the proof of assertion \ref{itstairs2}. One should apply the following equalities given by Theorem 2.1.1 and Proposition 2.3.3 of \cite{binters}: $$\obj\dmerb\lan a_{\ii,i}\ra \cap \obj d_{\le r}\dmerb =  \obj (d_{\le r-a_{\ii,i}}\dmerb)\lan a_{\ii,i}\ra$$ and $$\obj \chower\lan a_{\ii,i}\ra\cap \obj d_{\le r}\dmerb\break =\obj \chower\lan a_{\ii,i}\ra$$

\ref{itstairs5}. Obvious; see condition A in assertion \ref{itstairs2}.
%(one can also use Theorem 4.2.3 of $\cite{bos}$ or Proposition 3.1.1 of $\cite{bsnew}$).
\end{proof}

 \begin{rema}\label{rinfe} Taking $\ii=\z\times [0,+\infty)$ in our theorem we immediately obtain that any infinitely effective object of $\dmerpl$ is zero; cf. Remark \ref{bbard}(1).
\end{rema}

Now we relate our theorem to higher Chow-weight homology.

\begin{pr}\label{phcwh}
Let $\ii\subset \z\times [0,+\infty)$ and $M\in \dmerpl$. % (see Definition \ref{dwso}(\ref{lrbo})) be fixed.

Consider the following assumptions on $M$.
\begin{enumerate}
%\item\label{ir1} For some function $f_M:I\to  [0,+\infty)$  we have $\chw_{j-f_M(i,j)}^{i}(M_K,R,f_M(i,j))=\ns$ for all $(i,j)\in \ii$ and all function fields $K/k$.
\item\label{ir2}  $\chw_{j}^{i}(M_K,R)=\ns$ for all $(i,j)\in \ii$ and all function fields $K/k$.

\item\label{irrat} For all rational extensions $K/k$ and $(i,j)\in \ii$  we have $\chw_{j-1}^{i}(M_K,1)=\ns$.

\item\label{ir3}  $\chw_{0}^i(M_K,j)=\ns$ for all $(i,j)\in \ii$ and all function fields $K/k$.

\item\label{ir4} $\chw^i_a (M_K,j-a)=\ns$ for all $(i,j)\in \ii$, $a\in \z$, and all function fields $K/k$.  
\end{enumerate}

Then the following statements are valid.

1. Condition \ref{ir4} implies conditions \ref{ir3} and \ref{irrat},  and either of the latter two conditions implies  condition \ref{ir2}. 

2. Let $\ii$ be  a staircase set.  Then our conditions \ref{ir2}--\ref{ir4} are equivalent.
\end{pr}

\begin{proof} 
1. Obviously, condition \ref{ir4} implies all other ones. The proofs of the remaining implications are similar to that in Proposition 3.4.1 of \cite{bsoscwhn}.%Cite preliminaries???! (see also Remark 5.2.7(7) of \cite{bgn})

2. It remains to check that condition %For the remaining implication 
\ref{ir2} %$\Rightarrow$ 
 implies condition \ref{ir4}. %^, we %note that $M$ %motive as in 
 For this purpose we combine Theorem \ref{tstairs}(\ref{itstairs2}) (see Conditions A and D in it)  with %of Theorem \ref{tstairs} (2). Then our implication follows from 
 Proposition \ref{pr1}(1,4).
\end{proof} 

\begin{rema}
Note that the implication %$2 \Rightarrow 5$
 \ref{ir2} $\Rightarrow$ \ref{ir4}
 from Proposition \ref{phcwh} may be false if $\ii$ is not a staircase set.  For example, %consider the set 
 take $\ii = [2,+\infty) \times  [0,+\infty) \cup \ns \times [0,5]$. Then the motif $\mathbb{Q}\lan 1 \ra[-1]$ obviously satisfies the vanishing
 %conditions of $2$, 
 in condition  \ref{ir2}; yet $\chw^{2}_{1}(\mathbb{Q}\lan 1 \ra[-1],1)\cong \mathbb{Q}$. %This is one of the reasons for introducing the concept of staircase sets for our purposes.
\end{rema}

\subsection{On the relation to motivic homology}\label{smoth} %and slices??

Now we  define certain "steep" staircase sets and birational motives.  %introduce the following modernization of Definition $\ref{dreasi}$.

\begin{defi}\label{sstar}
1. Let $(i_0,j_0)\in \z\times  [0,+\infty)$. Then we define $S_{(i_0,j_0)}\subset \z\times  [0,+\infty)$ as the set $\{(i,j):\ i\ge i_0,\ 0\le j\le j_0+(i-i_0)\}$; %(%note that %see the figure 
%in \S\ref{sfig} below 
we illustrate this definition by marking in grey  the points of the  set $S_{(1,1)}$ %is drown 
 %are marked in grey on the picture in  \S\ref{sfig} below). %????????  % for the case $ \{i>j\ge 0\}\bigcup \{i+1>j\ge 0, i \geq 1\}$)). %picture??? where????????! 
on the following picture: 
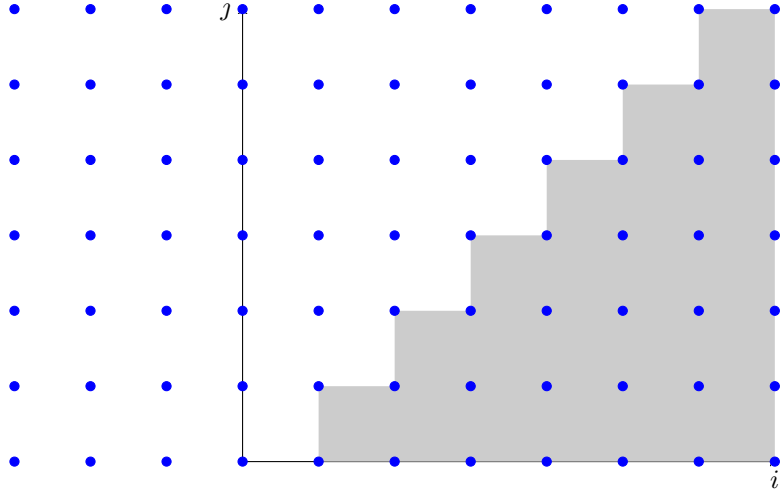
\begin{figure}[h!] %\label{Figure 2}
\center
%\caption %*{Figure }
\begin{tikzpicture}
\draw [<->] (0,6) node [left] {$j$} -- (0,0) -- (7,0)node [below] {$i$} ;
\fill[fill=gray!40] (1,0) --(1,1)-- (2,1)--(2,2)--(3,2)--(3,3)--(4,3)--(4,4)--(5,4)--(5,5)--(6,5)--(6,6) -- (7,6) -- (7,0) -- cycle;
\draw[fill][blue] (1,1) circle [radius=0.06];
\draw[fill][blue] (0,0) circle [radius=0.06];
\draw[fill][blue] (2,2) circle [radius=0.06];
\draw[fill][blue] (3,3) circle [radius=0.06];
\draw[fill][blue] (4,4) circle [radius=0.06];
\draw[fill][blue] (5,5) circle [radius=0.06];
\draw[fill][blue] (6,6) circle [radius=0.06];
\draw[fill][blue] (1,2) circle [radius=0.06];
\draw[fill][blue] (2,1) circle [radius=0.06];
\draw[fill][blue] (3,1) circle [radius=0.06];
\draw[fill][blue] (1,3) circle [radius=0.06];
\draw[fill][blue] (4,1) circle [radius=0.06];
\draw[fill][blue] (5,1) circle [radius=0.06];
\draw[fill][blue] (1,4) circle [radius=0.06];
\draw[fill][blue] (1,5) circle [radius=0.06];
\draw[fill][blue] (0,1) circle [radius=0.06];
\draw[fill][blue] (0,2) circle [radius=0.06];
\draw[fill][blue] (0,3) circle [radius=0.06];
\draw[fill][blue] (0,4) circle [radius=0.06];
\draw[fill][blue] (0,5) circle [radius=0.06];
\draw[fill][blue] (0,6) circle [radius=0.06];
\draw[fill][blue] (1,0) circle [radius=0.06];
\draw[fill][blue] (2,0) circle [radius=0.06];
\draw[fill][blue] (3,0) circle [radius=0.06];
\draw[fill][blue] (4,0) circle [radius=0.06];
\draw[fill][blue] (5,0) circle [radius=0.06];
\draw[fill][blue] (6,0) circle [radius=0.06];
\draw[fill][blue] (6,1) circle [radius=0.06];
\draw[fill][blue] (6,2) circle [radius=0.06];
\draw[fill][blue] (6,3) circle [radius=0.06];
\draw[fill][blue] (6,4) circle [radius=0.06];
\draw[fill][blue] (6,5) circle [radius=0.06];
\draw[fill][blue] (5,2) circle [radius=0.06];
\draw[fill][blue] (5,3) circle [radius=0.06];
\draw[fill][blue] (5,4) circle [radius=0.06];
\draw[fill][blue] (5,6) circle [radius=0.06];
\draw[fill][blue] (4,2) circle [radius=0.06];
\draw[fill][blue] (4,3) circle [radius=0.06];
\draw[fill][blue] (3,2) circle [radius=0.06];
\draw[fill][blue] (4,5) circle [radius=0.06];
\draw[fill][blue] (4,6) circle [radius=0.06];
\draw[fill][blue] (3,4) circle [radius=0.06];
\draw[fill][blue] (3,5) circle [radius=0.06];
\draw[fill][blue] (3,6) circle [radius=0.06];
\draw[fill][blue] (2,3) circle [radius=0.06];
\draw[fill][blue] (2,4) circle [radius=0.06];
\draw[fill][blue] (2,5) circle [radius=0.06];
\draw[fill][blue] (2,6) circle [radius=0.06];
\draw[fill][blue] (1,6) circle [radius=0.06];
\draw[fill][blue] (7,0) circle [radius=0.06];
\draw[fill][blue] (7,1) circle [radius=0.06];
\draw[fill][blue] (7,2) circle [radius=0.06];
\draw[fill][blue] (7,3) circle [radius=0.06];
\draw[fill][blue] (7,4) circle [radius=0.06];
\draw[fill][blue] (7,5) circle [radius=0.06];
\draw[fill][blue] (7,6) circle [radius=0.06];
\draw[fill][blue] (-1,0) circle [radius=0.06];
\draw[fill][blue] (-1,1) circle [radius=0.06];
\draw[fill][blue] (-1,2) circle [radius=0.06];
\draw[fill][blue] (-1,3) circle [radius=0.06];
\draw[fill][blue] (-1,4) circle [radius=0.06];
\draw[fill][blue] (-1,5) circle [radius=0.06];
\draw[fill][blue] (-1,6) circle [radius=0.06];
\draw[fill][blue] (-2,0) circle [radius=0.06];
\draw[fill][blue] (-2,1) circle [radius=0.06];
\draw[fill][blue] (-2,2) circle [radius=0.06];
\draw[fill][blue] (-2,3) circle [radius=0.06];
\draw[fill][blue] (-2,4) circle [radius=0.06];
\draw[fill][blue] (-2,5) circle [radius=0.06];
\draw[fill][blue] (-2,6) circle [radius=0.06];
\draw[fill][blue] (-3,0) circle [radius=0.06];
\draw[fill][blue] (-3,1) circle [radius=0.06];
\draw[fill][blue] (-3,2) circle [radius=0.06];
\draw[fill][blue] (-3,3) circle [radius=0.06];
\draw[fill][blue] (-3,4) circle [radius=0.06];
\draw[fill][blue] (-3,5) circle [radius=0.06];
\draw[fill][blue] (-3,6) circle [radius=0.06];
\end{tikzpicture}
\end{figure}

2. Let $\ii$ be a subset of $\z\times  [0,+\infty)$ (see \S\ref{snotata}).\\
We will call it a {\it superstaircase} set if for any $(i_0,j_0)\in \ii$ we have  $S_{(i_0,j_0)}\subset \ii$.
 %and $i\ge i_0$, and $0\le j\le j+(i-i_0)$ the point $(i,j)$ belongs to $\ii$ as well.
 %it  is %equals %the union of the corners
 %the union of sets of the form $\bigcup \limits_{s\ge 0} \{(i,j):\ i+t_{s}>j\ge 0, i\geq n_{s}\}$ for some $n_{s}\geq 0, t_{s}\in \mathbb{Z}$.%\{(i,j):\ i>j\ge 0\} \cup \bigcup 

3. We will say that an object $N$ of $\dmer$ is {\it birational} if  $\obj \dmer\lan 1 \ra\perp N$.
\end{defi}

Let us make some observations related to superstaircase sets and slices. 

\begin{rema}\label{rss}
1. Obviously, any superstaircase is a staircase one.

Moreover, for if $(i,j)\in S_{(i_0,j_0)}$ then $S_{(i,j)}\subset S_{(i_0,j_0)}$. It clearly follows that a subset of $\z\times  [0,+\infty)$.
is superstaircase if (and only if) it can be presented as the union of    "sectors"  $S_{(i_l,j_l)}$ for some $(i_l,j_l)\in S_{(i_0,j_0)}$.

2. %We have to recall some facts about slices; see section 3 of \cite{kabir} (where the existence of certain adjoint functors is justified).\\
Consider the %inclusion functor
 embedding $\dmerb \lan 1 \ra \to \dmerb$; the composition $\nu^{\geq1}$ of this inclusion with its right adjoint can clearly be described as %$\nu^{\geq1}(-)=
 $\underline{Hom}(R\lan 1 \ra, -)\lan 1 \ra$; see Proposition 4.6.2 of \cite{kabir}. %Next, we %denote 
 %We will write $\nu^{0}$ for the composition of the localisation $\dmerb \to \dmr^0$ with its right adjoint. If an object $M$ of $\dmer$ belongs to the image of $\nu^{0}$ then we will say that it is {\it birational}. This is easily seen to be equivalent to $\obj \dmer\lan 1 \ra\perp M$; see Lemma 4.5.4 of \cite{kabir}. see Definition \ref{sstar}(3)
%We will say that an object in the image of $\nu^{0}$.

Moreover, for any object $M$ of $\dmer$  Proposition 4.6.2 of \cite{kabir} gives the following {\it slice filtration} triangle
\begin{equation}\label{etr} % \nu^{0}(M)[-1] \to
 \nu^{\geq1}(M) \to M \to \nu^{0}(M)\to \nu^{\geq1}(M)[1].  \end{equation}

Recall also that the object $\nu^{0}(M)$ is birational in the sense of Definition \ref{sstar}(3); see  Lemma 4.5.4 of ibid.
%see Remark \ref{rss}(2) and Proposition 4.6.2 of \cite{kabir}.
\end{rema}

Now we relate the Chow-weight homology %with 
 to higher motivic homology.

\begin{theo}\label{main}
Let $\ii=%\bigcup \limits_{s\ge 0} \{(i,j):\ i+t_{s}>j\ge 0, i\geq n_{s}\}
\cup S_{(t_{s},n_s)}$ for some $ t_{s}\in \mathbb{Z},\ n_{s}\ge 0$. Then the following assumptions on $M\in  \dmerpl$  are equivalent:

(a). $\chw_{j}^{i}(M_K)=\ns$ for all $(i,j)\in \ii$ and all function fields $K/k$;

(b). $\chowm_{r}(M_K, R, -c)=\ns$ for all these $K$, $0\leq r \leq n_{s}$, and $c \ge t_{s}$;

 (c) $\chowm_{r}(M_K, R, -c)=\ns$ for all these $K$ and $(r,c)\in \ii$.
%if and only if $\chowm_{r}(M_K, R, c)=\ns$ for all these $K$,  $0\leq r \leq n_{s}+ t_{s}-1$,  and $c \leq- n_{s}$. %; 0 \leq c \leq (max(t_{s}) - 1) or t_{s}-1-r ?$.
\end{theo} %Some extension-closure????

\begin{proof}
Clearly, Condition (c) implies Condition (b). 

Next, assume that Condition (a) is fulfilled and $(i,j)\in \ii$. Since $\ii$ is a superstaircase set, $(u,j+u-i)\in \ii$ whenever $u\ge i$.
Hence %$\chw_{a}^{u}(M_K,R,q-a)=\ns$ whenever $(a,u+a-q)\in \ii$ 
 %$\chw_{j}^{i}(M_K,R,q-a)=
$\chw^u_j (M_K,{u-i})= \ns$ for any %$u\in \z$
 $u\ge i$   by   Proposition \ref{phcwh}; see conditions \ref{ir4} and \ref{ir2} in it.

%Assume that $M$ satisfies the corresponding Chow-weight homology vanishing assumptions.  
Next, the functor $H=\chowm_{j}(-_K,R,{u-i})$ kills $\dmer_{\wchow\ge 1}$ for all $u\in \z$ and vanishes if $u<i$  % and $\chw^a_b (M_K,{c})=\ns$ if $a,b\in \z$ and $c<0$
 %since the  base field change functor $\dmer\to \dmerb(K^{perf})$ is weight-exact
 by Proposition \ref{pr1}(4). Thus we have a converging Chow-weight spectral sequence %$T(H,M)$ 
$$T(H,M): E_2^{u,q}T(H,M)  =\chw_{j}^{u}(M_K,R,-q)\implies E_{\infty}^{u+q}=\chowm_{j}(M_K, R, -u-q)$$ that corresponds to $\wchow$; % and the homological functor $\chowm_{a}(-_K,R)$; 
 see Proposition \ref{pwss}.  Applying the aforementioned vanishing of the corresponding groups $\chw^u_j (M_K,{u-i})$ %implies 
 we obtain $\chowm_{j}(M_K, R, -i)=\ns$; thus  Condition (a) implies Condition (c). 

 %$\chw^a_b (M_K,{c}) $ %Since $\chw_{j}^{u}(M_K,R,-q)=\ns$ for??? see?? %????(1). 

It remains to verify that (b) implies (a). Since $\ii=\cup S_{(t_{s},n_s)}$, it suffices to prove that this implication is valid for $\ii= S_{(t,n)}$, where $ t\in \mathbb{Z}$ and $ n\ge 0$. Moreover, we can clearly assume that $t=0$. Hence it remains to prove the following lemma.

\begin{lem}\label{hhom}
Assume that $n\ge 0$, $M\in \obj \dmer$, %??{}_{\wchow+}$, %not necessary????!!!
 and $\chowm_{r}(M_K, R, -c)=\ns$  if $c\ge 0$, $0 \le r\le n$, and $K/k$ is a function field. 

Then $\chw_{j}^{i}(M_K)=\ns$ for all $(i,j)\in S_{(0,n)}$ (and all function fields $K/k$). 
%Then $\chowm_{r}(M_K, R, -c)=\ns$ (for these $c,r,K$) as well.

%$M\in \obj \dmerb$ be fixed, $\ii$  our theorem (see the figure in \S\ref{sfig} below for the case $ \{i>j\ge 0\}\bigcup \{i+1>j\ge 0, i \geq 1\}$). %\ref{Figure 2} 
%2 %below). \\  at the end of the paper). 
% If $\chowm_{r}(M_K,R,c)=\ns$ for all $0\leq r \leq n_{s}+ t_{s}-1, c \leq -n_{s}$, and all function fields $K/k$, then $\chw_{j}^{i}(M_K)=\ns$ for all $(i,j)\in \ii$ and all function fields $K/k$.
\end{lem}

\begin{proof} Let us prove the statement by induction on $n\ge 0$.

In the case $n=0$ the statement is a simple generalization of Corollary 3.4.2 of \cite{bsoscwhn}; we essentially repeat the argument here. By %Theorem 3.3.1 of \cite{bondegl} %(cf. also  Corollary 4.18 of \cite{3}), 
 %(see also Definition 3.2.3, \S1.3.8, and the formula  (1.3.7.c) of ibid.; cf. \S3.1 of \cite{degmod}) %below??????!
 %Proposition 2.3.3(3) of  \cite{bsoscwhn}
 Proposition \ref{pmot}(\ref{ie3}) below, our assumption implies that $M$ belongs to $ DM^{eff,\thomr \le -1}_{R}$ (see Definition \ref{dcwh}(4)).
Thus it remains to apply Proposition \ref{pr1}(6). % (note that  the homotopy $t$-structure $\thomr$ is recalled there as well). %the definition of . 

Now we assume that the statement in question is valid if $n\le m$ for some $m\ge 0$. We should verify it for $n=m+1$. Since we have just proved %the statement in question 
 it in the case $n=0$, it suffices to verify that $\chw_{j}^{i}(M_K)=\ns$  for  $(i,j-1)\in S_{(0,n-1)}$. %???

%consider 

 %Taking the corresponding 
We take the slice filtration distinguished triangle (\ref{etr}) and denote  $\nu^{0}(M)$ and $\nu^{\geq1}(M)$ by $M^0$ and $M^1$, respectively.  Then for any $i\in \z$, $j\ge 0$, and function field $K$ %the long exact sequence for Chow-weight homology over a function field $K$ coming from
 (\ref{etr}) gives a long exact-sequence  $\dots \to \chw_{j}^{i}(M^1_K)\to \chw_{j}^{i}(M_K)\to \chw_{j}^{i}(M^0_K)\to \dots$ for  Chow-weight homology. %we obtain that 
 Thus it suffices to verify that
$\chw_{j}^{i}(M^0_K)=\chw_{j}^{i}(M^1_K)=\ns$ %for all $(i,j)\in S_{(0,n)}$
 whenever  $(i,j-1)\in S_{(0,n-1)}$. %, where $M^0= \nu^{0}(M)$ and $M^1= \nu^{\geq1}(M)$. 

Now, it is easily seen that for any $i\in \z$ and $j\ge 1$ we have  $\chowm_{j}(M^1_K, R, -i)=\chowm_{j}(M_K, R, -i)$ and  $\chowm_{j}(M^0_K, R, -i)=\ns$ (cf. %Definition 3.2.3 and?? Theorem 3.3.1 
 the aforementioned statements from   \cite{kabir} once again). Thus for $M^2=\underline{Hom}(R\lan 1 \ra, M)$ (see  Remark \ref{rss}(2))  we have $\chowm_{j-1}(M^2_K, R, -i)=\ns$ whenever  $i\ge 0$ and  $j-1\le n-1$. Applying the inductive assumption we obtain 
$\chw_{j-1}^{i}(M^2_K)=\ns$ if $(i,j-1)\in S_{(0,n-1)}$. Now, Proposition \ref{pwt}(\ref{iwcpu}) easily implies that $\chw_{j}^{i}(-_K)\circ \lan 1\ra\cong \chw_{j-1}^{i}(-_K)$; hence  $\chw_{j}^{i}(M^1_K)=\ns$ if  $(i,j-1)\in S_{(0,n-1)}$ (recall that $M^1\cong M^2\lan 1\ra$).  Moreover,  $M^1$ belongs to $ \dmer^{\thomr\le -1}$ since $M$ does; see Corollary 3.3.7(2) (along with Theorem 3.3.1) of \cite{bondegl} or Proposition \ref{pmot}(\ref{ie3}) below.  %combining the long exact sequence for the functors  $\chowm_0{(-_K)}$
%Hence the statement follows from .
 It easily follows that $M^0$ belongs to $ \dmer^{\thomr\le -1}$ as well. Since $M^0$ is birational, it belongs to  $\dmer_{\wchow\ge 1}$ by Lemma \ref{lem}(2) below. Thus $\chw_{j}^{i}(M^0_K)=\ns$ whenever $i\ge 0$ (see Proposition \ref{pr1}(4)) and we can conclude the proof. 
\end{proof}

\begin{lem}\label{lem} %Assume that $N$ is  birational (see  Definition \ref{sstar}(3)) and belongs to $\dmer^{\thomr\le 0}$. Then $M$ belongs to  $\dmer_{\wchow\ge 0}$.
 Let $N\in \dmer^{\thomr\le 0}$.

1. Then for any $j\ge 0$ there exists a choice of $w_{\chow\le -j}N$ that belongs to $\obj \dmer\lan j\ra$.

2. Assume in addition that $N$ is  birational in the sense of   Definition \ref{sstar}(3). Then $M\in \dmer_{\wchow\ge 0}$.
\end{lem}

\begin{proof} %Let us prove that there exists a choice of $\wchow\ge 0}$
1. If $N\in  \dmerpl$ then the statement is an easy combination of the case $n=0$ of Lemma \ref{hhom} (note that this case of that lemma does not depend on our one) with  Theorem \ref{tstairs}(\ref{itstairs2}) (see conditions A and C in it).

The argument for the general case is similar to the proof of \cite[Proposition 2.3.2(10)]{bwcomp}. Let us fix $j\ge 0$.  Since the weight structure $\wchow$ is smashing, the class $C$ of those $M\in \obj \dmer$ such that there exists  $w_{\chow\le -j}M$ that belongs to $\obj \dmer\lan j\ra$ is smashing (see Definition \ref{desmash}(\ref{idsmash}))  in $\dmer$; see Proposition 2.3.2(3) of loc. cit. Moreover, %this class
 $C$ is extension-closed by Proposition 1.2.4(12) ibid. Recalling Proposition \ref{pr1}(6) once again we obtain that is suffices to verify that there exists $w_{\chow\le -j}M$ whenever $M\in \obj \chower \lan a\ra [b-a]$ for $a,b\ge 0$, and this is obvious.

2. %We take %the corresponding
 According to the previous assertion, there exists a  $-1$-weight decomposition triangle $w_{\chow\le -1}N\stackrel{a}{\to} N\to w_{\ge 0}N$ (see Remark \ref{rstws}(2)) such that  $w_{\le -1}N\in \obj \dmer\lan 1\ra$. Since $N$ is birational,  $w_{\le -1}N\perp N$ (see Remark \ref{rss}(2)). Hence $N$ is a retract of $w_{\ge 0}N\in \dmer_{\wchow\ge 0}$; thus $N$ belongs to $ \dmer_{\wchow\ge 0}$ itself indeed.
\end{proof}

Thus, we %obtain the statement of 
 have completed the proof of  Theorem \ref{main}.
\end{proof}

\begin{rema}\label{comt}

1. Below we will mostly apply Theorem \ref{main} to geometric motives. Note however that our argument relies on slices; thus one cannot "apply it inside $\dmger$" (see \cite{ayapp}).

Moreover, the proofs of Theorems \ref{tstairs} and \ref{main} hint that it can make sense to study conditions of these theorems for subcategories of $\dmer$ that are bigger than $\dmerpl$. In particular, one may treat the Voevodsky category $\dmerm\supset \dmerpl$; cf.  \S2.3 of \cite{binters}.

2. %Our arguments easily yield that for any %super???!!
Since the slice functors are exact and respect coproducts, for any staircase set $\ii$ an object $M$ of $\dmer$ %??_{\wchow+}$% and any superstaircase set $\ii$ we have $M\in
 belongs to the class $ D_{\ii}$ (see condition D in Theorem \ref{tstairs}(\ref{itstairs2})) if and only if  $\nu^{0}(M)$ and $\nu^{\geq1}(M)$ do. Moreover, similar implications hold for other "slices" of $M$. We leave the detail for these statements to the reader.

3. Obviously, there are plenty of staircase subsets of $\z\times [0,+\infty)$ that are not superstaircase ones. However, the only "concrete" staircase of this sort that were considered in \cite{bsoscwhn} are sets of the form $\z\times [0,c-1]$ (for $c>0$). They correspond to the $c$-effectivity of motives; cf. Remark 3.3.2(2) of ibid. %One can avoid CWH??

4. Now we demonstrate that Theorem \ref{main} %our statements may fail  if $\ii$ is not a staircase set. 
 does not extend to the case where $\ii$ is an arbitrary subset of $\z\times [0,+\infty)$. Let $R=\q$ and $M=\mathbb{Q}\lan 1\ra[-1]$.
Then $\chw_{j}^{i}(M)=\ns$ for $(i,j)\neq (1,1)$. 
 
Next, assume that $k$ is not a union of finite fields, and $\ii=[0,+\infty)\times [0,+\infty)\setminus \{(1,1)\}$. %  $K$ is an infinite extension of $k$, and  %\{(i,j):\ i>j\ge 0\} \bigcup \{(i,j):\ i+2>j\ge 2\}$. Consider the shifted Tate motif 
  %??????? %K=??
Then $\chw_{0}^{0}(N, \mathbb{Q})=k^{\times}\otimes \q\neq\ns$;  yet $(0,0)\in \ii$. % for all $(i,j) \in \ii$; yet clearly $\chowm_{1}(N,-1)\neq 0$.

5. Assume  $M\in  \dmerpl$, $t\ge 0$, %. Then  and that for
 and for any $t\ge 0$ the $E_{2}^{*,*}-$terms of the Chow-weight spectral sequence $T_{\wchow}(H,M)$ for the homological functor $\dmerb(R\lan t \ra,-)$ are concentrated in the first quadrant (in particular, this is the case if  $M\in \dmer_{\wchow\le 0}$). % ; cf. also Remark 5.4.1(6) of \cite{bsoscwho}).
 Thus  we have the so-called five-term exact sequence: 
$$ 0 \to \chw_{t}^{1}(M) \to \chowm_{t}(M,-1) \to \chw_{t}^{0}(M,-1) \to \chw_{t}^{2}(M) \to \chowm_{t}(M,-2).$$
%$M\in \dmer_{\wchow\le 0}$. %Consider 
% Then we have the so-called five-term exact sequence, coming from the Chow-weight spectral sequence $T_{\wchow}(H,M)$ for the homological functor $\dmerb(R\lan t \ra,-)$:
%$$ 0 \to \chw_{t}^{1}(M) \to \chowm_{t}(M,-1) \to \chw_{t}^{0}(M,-1) \to \chw_{t}^{2}(M) \to \chowm_{t}(M,-2).$$
%Thus % we could 
  Clearly one may obtain some homology vanishing statements %conditions for  of (low-term) homologies 
	from this sequence.%\\
\end{rema}

\section{Applications to geometric motives}%and supplements}
\label{sappl}

In this section we %combine Theorem \ref{main} with some results of  \cite{bsoscwhn}; recall that there geometric motives were treated only.
apply  Theorem \ref{main} to obtain some new statements on geometric motives. We argue similarly to \cite{bsoscwhn}.

In \S\ref{sexp} we combine our Theorem \ref{main} with the results of  \cite[\S3.6]{bsoscwhn}; this roughly gives the finiteness of exponents of higher Chow-weight homology and lower motivic homology groups provided that they are torsion.

In \S\ref{samgc} we apply our results to the motif with compact support of a variety $X$. It follows that if certain Chow homology groups of $X$ (over a universal domain $K$ containing $k$) are torsion then they are of finite exponent, and also %gives 
 estimates the effectivity of the corresponding Deligne weight factors of singular and \'etale cohomology. 

\subsection{%Applications to geometric motives
On Chow-weight and motivic homology of bounded exponent}\label{sexp}

Now we apply our results to geometric motives.

\begin{theo}\label{trs}
Let $M \in Obj DM^{eff}_{gm, \mathbb{Z}[\frac{1}{p}]}$, $K$ be a universal domain (that is, $K$ is an algebraically closed field that is of infinite transcendence degree over its prime subfield) containing $k$, and $\ii=\cup S_{(t_{s},n_s)}$ %is a superstaircase set
  (see Definition \ref{sstar}(1)).

The following conditions are equivalent. % (see Definition \ref{sstar} for the notation).

\begin{enumerate}
\item\label{trhc0} 
%$\chowm_{r}(M_K, c, \mathbb{Q})=\ns$ for all $0\leq r \leq n_{s}+ t_{s}-1, c \geq -n_{s}$.
$\chowm_{r}((M\otimes \q)_K, \q, -c)=\ns$ for  $0\leq r \leq n_{s}$ and $c \ge n_{s}$; here $M\otimes \q$ is the result of the application to $M$ of the extension of scalars functor $-\otimes \q=-\otimes_{\zop} \q:\dmgep\to \dmgeq$ provided by Proposition 3.6.2(I.1) of \cite{bsoscwhn}.
%\item\label{trhc1} $\chowm_{r}(M_{k'}, c, \mathbb{Q})=\ns$ for all $0\leq r \leq n_{s}+ t_{s}-1, c \geq -n_{s}$, and all field extensions $k'/k$.
\item\label{trhc2}
There exists $E_{M}>0$ such that $E_{M}\chowm_{r}(M_{k'}, -c, \mathbb{Z}[\frac{1}{p}])=\ns$ for any $(r,c)\in \ii$ and any field extension $k'/k$.
%for all $0\leq r \leq n_{s}+ t_{s}-1, c \geq -n_{s}$.

\item\label{trhc3}  $\chw_{j}^{i}((M\otimes \q)_K,\mathbb{Q})=\ns$ whenever $(i,j)\in \ii$. %$a\in \mathbb{Z}$, $i+t_{s}>j$, and all field extensions $k'/k$.

\item\label{trhc4}
 There exists $E'_{M}>0$ such that $E'_{M}\chw_{j-a}^{i}(M_{k'},a)=\ns$ for all $a\in \mathbb{Z}$, %$i+t_{s}>j$
$(i,j)\in \ii$, and all field extensions $k'/k$. %Global constant?????????????????????
\end{enumerate}
\end{theo}

\begin{proof}
%Lastly, c
Condition \ref{trhc2} obviously implies condition \ref{trhc0}.

Next  Proposition 2.3.4(II) of \cite{bsoscwhn} implies that condition \ref{trhc0} (resp. \ref{trhc3}) is fulfilled if and only if we have similar vanishing over any field extension $k'/k$. Hence these conditions are equivalent according to Theorem \ref{main} applied to the motif $M\otimes\q\in \obj \dmeq$.

Moreover, conditions \ref{trhc3} and \ref{trhc4} are equivalent according to Theorem 3.6.4(I,II) of ibid. (see condition II.B in it).

Lastly we argue similarly to the proof of Corollary 3.6.5(II) of ibid. As we have already noted in the proof of Theorem \ref{main}, we have a convergent Chow-weight spectral sequence $T(H,M)$ as follows: $$ E_2^{u,q}T(H,M)  =\chw_{j}^{u}(M_K,\zop,-q)\implies E_{\infty}^{u+q}=\chowm_{j}(M_K, \zop, -u-q).$$ %Since $M$ is $\wchow$-bounded, 
 Next, the complex $t^{st}(M)$ is isomorphic to a bounded one (see Definition 3.1.1(1) and Proposition 2.2.1(1) of \cite{bsoscwhn}); hence  the simple %indices??? 
 index computation made above yields that for $(r,c)\in \ii$ the group $\chowm_{r}(M_{k'}, -c, \mathbb{Z}[\frac{1}{p}])$ possesses a filtration of a uniformly bounded length whose factors are killed by the multiplication by $E'_{M}$. Thus we can take $E_M$ to be a high enough power of $E'_{M}$.
 %$E_2^{pq}(T(M,k'))=\chw^p_{0}(M_{k'},-q,\zop)$
%a simple spectral sequence argument demonstrates that condition \ref{trhc4} implies condition \ref{trhc2}; see the proof of Corollary 3.6.5 of ibid. or
 %The converse implication can be proved by %Our conditions \ref{trhc0} ,\ \ref{trhc1} , and  \ref{trhc3} are equivalent by Theorem \ref{main}. The remaining implications follow from Proposition 3.6.5 and Corollary 3.6.6 of \cite{bsoscwhn}.
\end{proof}

\begin{rema}\label{rweirdtorsion}
It is  quite remarkable %that we can %can prove the boundedness of 
%bound exponents of 
 that certain Chow-weight homology %groups have
 has finite exponents. %since a priori nothing prevents the 
Note that (in general) Chow-weight homology groups and %well as %the more popular motivic homology ones that we also treat above) 
motivic homology of geometric motives can   have really "weird" torsion.
\end{rema}

\subsection{%Chow-weight homology for
Applications to motives with compact support}\label{samgc}

%In this section, we apply 
Let us apply our results to motives with compact support. Let us recall some basics on these motives.
%For $R$-linear coefficients? Just state the corresponding version of Theorem 4.2.1???

\begin{pr}\label{pmgc}

1. There exists a functor $\mgcr$ of the motif with compact support from the category $\schpr$ of $k$-varieties with morphisms being proper ones into $\dmger$. %that is  provided by  \S4.1 of \cite{1} along with \S5.3 of \cite{kellyast}, satisfies the following properties. 

%\item\label{imchow}
2. For any $j,l\in \z$, $X\in \var$, $M=\mgcr(X)$, and any field extension $k'/k$ the group $\chowm_j(M_{k'},R,l)$ is
naturally isomorphic to the higher Chow group $\chowgri_{j}(X_k',l,R)$  (cf. Theorem 5.3.14 of \cite{kellyast} for the $R=\zop$-version of this notation). 
%\end{enumerate}
\end{pr}
\begin{proof}
These statements easily follow from their  $\zop$-linear versions provided by %(and of) 
 \S5.3 of \cite{kellyast} along with (Proposition 1.3.3 of \cite{bokum} and) % Proposition 4.1.8(1) of \cite{bsoscwhn} 
Proposition \ref{pmot}(\ref{ie2}) below; cf. Proposition 4.1.8(1) of \cite{bsoscwhn}. 
\end{proof}

\begin{rema}\label{rmgq}
 %Actually, 
 Recall that actually the functor $\mgr$ is defined on the category of all $k$-varieties, and we have $\mgr(X)=\mgcr(X)$ whenever $X$ is proper (see Proposition 5.3.5 of \cite{kellyast}).  In particular, $\mgr(X)=\mgcr(X)$  if $X$ is smooth projective.
\end{rema}

Now we combine Theorem \ref{main}  with certain results of \cite{bsoscwhn} to obtain an extension of Theorem 4.2.1 of ibid.  Note  that  this statement does not mention Chow-weight homology.

\begin{theo}\label{tmgc}
Assume that  $X \in \var$, $K$ is a universal domain containing $k$, and for some set of %$ t_{s}\in \mathbb{Z},\ n_{s}\ge 0$
 $\{(t_s,n_s)\}\subset \z\times [0,+\infty)$ we have $\chowgri_{r}(X_{K},-c, \mathbb{Q})=\ns$ %if 
  whenever there exists $s$ such that $0\leq r \leq n_{s}$ and $c \ge t_{s}$. 

 1. Then %the following statements are valid.
%\begin{enumerate}
%\item\label{obv} There 
 there exists $E>0$ such that $E\chowgri_{r}(X_{k'},-c, \zop)=\ns$ for all $(r,c) \in \ii$ and any field extension $k'/k$, where $\ii=\cup S_{(t_{s},n_s)}$ (see Definition \ref{sstar}(1)). %  S_{(0,n)}$.

2. If $k$ is a subfield of $\com$  and $l,m\in \z$ then the $m+l$-th (Deligne) weight factor of $H^{m}_c(X_{\com})$ of the ($\q$-linear) singular cohomology of $X_{\com}$ with compact support   is $a_{\ii,l}$-effective %(see Definition \ref{dreasi}(1))
  as a pure Hodge structure; see Definition \ref{dreasi}(1) above %, and Definition 3.5.3 and Theorem 3.5.4(2) of 
	and \cite[Definition 3.5.3, Theorem 3.5.4(2)]{bsoscwhn} for the corresponding definitions. 

Moreover, the same effectivity properties hold for Deligne weight factors of $\ql$-\'etale cohomology $H^{m}_c(X_{k^{alg}})$ %is fulfilled if $k$ is an
 if $k$ is the perfect closure of a field that is of finite transcendence degree over its prime subfield.
%\end{enumerate}
\end{theo}

\begin{proof}
1. Immediate from Proposition \ref{pmgc}(2) combined with Theorem \ref{trs} (see conditions \ref{trhc0} and \ref{trhc2} in it).  

2.  %The corresponding Chow-weight 
For $M=\mgcq(X)$ we have   $\chw_{j}^{i}(M_{K})=\ns$ if $(i,j)\in \ii$; see condition \ref{trhc4} of that theorem (or Theorem \ref{main}). Given this statement, one can argue similarly to Theorem 4.2.1(I.2) of \cite{bsoscwhn}.
\end{proof}

\begin{rema}\label{rlast}
1. %More generally
Moreover, for  $X$ as above,  $M= \mgcq(X)$, and any cohomological functor $H$ from %$DM^{eff}_{gm,\mathbb{Z}[\frac{1}{p}]}$
 $\dmgeq$  into %a $\mathbb{Q}-$linear
 an  abelian category $\underline{A}$ one can combine the Chow-weight homology vanishing mentioned in the proof with %and any $q>0$ %we have
Proposition 3.5.1(1)  of ibid. to obtain the following: for any $l,m\in \z$ both $E_2^{-l,m}T(M)$ %$E_2^{-sq}T(M)$ %_{\wchow}(H,M)$
 and %the quotient %object
  $(\grwc^{-l}H^{m-l})(M)$ %\linebreak   (W^{-l}H^{m-l})(M)/(W^{1-s}H_{s+q})(M)$ are certain
	 are subquotients of $H^m(\mgr(P)\lan a_{\ii,l}\ra)$ for some $P\in \spv$ whenever $a_{\ii,l}<+\infty$, and these two objects vanish if $a_{\ii,l}=+\infty$; see %Proposition \ref{pwss} 
	 Definition 1.4.4(3) and Proposition 1.4.5(2) of ibid. for the corresponding notation.

2.  %Recall that the 
 The combination of two of more or less "standard" %) 
 motivic conjectures yields that the %first  implication in Theorem \ref{tleci}(2) is  actually an equivalence. 
 Hodge effectivity condition in Theorem \ref{tmgc}(2) is actually equivalent to our assumptions on Chow groups of $X$. This statement is an easy implication of Proposition 3.5.6 of \cite{bsoscwhn}  (combined with our Theorem \ref{main}).

3. %Moreover, one 
 One can certainly consider Chow-weight spectral sequences for non-geometric objects of $\dmeq$ (or $\dmer$ for any $R$). In particular, one may extend to $\dmeq$ singular and \'etale homology functors similar to the ones mentioned in Theorem \ref{tmgc}(2). Note here that there exist homological functors of this sort that take values in the corresponding ind-completed categories and respect coproducts; see Lemma 2.2 of \cite{krause}.

An important observation here is that these functors convert objects of  $\choweq{}^{\widehat\oplus}$ into (ind-pure) objects of weight $0$
in the corresponding mixed categories; hence these weight spectral sequences degenerate at $E_2$ (cf. Theorem 3.5.4 of \cite{bsoscwhn}). \end{rema}

\appendix
\section{Some motivic statements}\label{sapp}

Let us prove some statements that were used both in \cite{bsoscwhn} and in the current paper. The authors were not able to find these formulations in the literature; yet no originality is claimed. 

We will not introduce any notation or definitions that will be used below;  most of it can be found in \cite{cdint}. %we will only give references to it.

\begin{prop}\label{pmot}
Let $K/k$ be an extension of perfect fields, $f:\spe K\to \spe k$ is the corresponding morphism, and $X$ is a $k$-variety.

Then the following statements are valid.

\begin{enumerate}
\item\label{ie1} The functor $-_K$ in  Definition \ref{dchowm}(2)  is essentially the restriction to $\dmer$ of the functor $f^*:\dmr\to \dmrbl$, and we have  $f^*(\mgr(X))\cong \mgr(X_{K})$. %, where $f^*$ is the corresponding  base field change functor (see Definition \ref{dchowm}(2)).
 %(cf. Proposition \ref{pr1}(1)).

\item\label{ie2} $f^*(\mgrc(X))\cong \mgrc(X_{K})$.

\item\label{ie3}
For an object $N$  of $ \dmer$ we have $N\in \dmer^{\thomr \le 0}$ %(see %the end of \S\ref{smotnot}
%Remark \ref{rhomr}) 
if and only if  
%extend??!
 $\chowm_{0}(N_{k'},R,l)=\ns$ for all $l<0$ and all %$K$ essentially finitely generated 
function fields  ${k'}/k$. 

Moreover,  these conditions are equivalent to the vanishing of   $\chowm_{r}(N_{k'},R,l-r)$ for all $l<0$, $r\ge 0$,  and all %$K$ essentially finitely generated?! N_K*????
function fields  ${k'}/k$; recall here that we identify $\dmer(k')$ with $\dmer(k'{}^{perf})$, where $ k'{}^{perf}$ is the perfect closure of $k'$. 
%uh-invariance of Chow groups along with continuity?!
\end{enumerate}
\end{prop}
\begin{proof}
\ref{ie1}. We pass to the "stable" motivic category $\dmrb\cong \dmcdh(\spe k,R)\supset \dmer$  (see %\S1.2,
 Remark 1.2,  Definition 1.5, and Proposition 8.1(c) of \cite{cdint}). Then the second part of the assertion implies that we can define 
 $-_K$ as the restriction to $\dmer$ of the functor $f^*$ indeed.

 To obtain the isomorphism in question we recall that $\mgr(X)$ can be computed as $x_!x^!(R)$, where $x:X\to \spe k$ is the structure morphism; see the formula (8.7.1) (and \S1.6)  of ibid.  
We take the corresponding Cartesian square
\begin{equation}\label{ebch} 
  \begin{CD}
 X_K@>{f_X}>>X\\
@VV{x_K}V@VV{x}V \\
\spe K@>{f}>>\spe k
\end{CD}\end{equation}
%Now we 
 and recall that the categories  $\dmcdh(-,R)$ give a {\it motivic category} over  the category of noetherian $k$-schemes of finite dimension  %{\it which has the property of continuity with respect to arbitrary projective systems with affine transition maps}
 that is {\it continuous with respect to the twists $\lan n\ra$} for $n\in \z$; see Proposition 4.3 and Theorem 5.1 of ibid.  for this statement,  and %D4.3.2 Example 1.4.5: $\z$-generated?! Definition 1.1.41?!
 Definitions 2.4.45 and 4.3.2 of \cite{cd} for the corresponding  definitions. 
Thus we can apply the base change isomorphism (see Theorem 2.4.50(4) of ibid.)  to obtain $f^*x_!\cong x_{K!}f_X^*$. 

Next, 
%Since 
 the motivic category $\dmcdh(-,R)$ is {\it generated by the aforementioned twists} (see Definition 2.5 and Proposition 4.3 of \cite{cdint}) and the morphism $f$ is {\it regular} immediately from the Popescu-Spivakovsky theorem (see Theorem 4.1.5 of \cite{cd}). Thus we can apply %Propositions 4.3.13 and 4.3.15 
 Propositions 4.3.12 of ibid. to obtain $f_X^*x^!\cong x_{K}^!f^*$. It remains to recall that $f^*R_{\spe k}\cong R_{\spe K}$ (see \S1.1.1 of ibid.); hence $f^*(\mgr(X))\cong x_{K!} x_{K}^! R_{\spe K} \cong \mgr(X_{K})$ indeed. 

%Now we consider the obvious 

\ref{ie2}. The %argument is just a little different 
 proof differs just a little from the previous one. Proposition 8.10 of \cite{cdint} says that $\mgrc(X)$ can be computed as $x_*x^!(R)$; here we use the notation of  \cite{cd} %; respectively, we 
 (and write $x_*$ instead of $Rx_*$). %As we have explained above, 
Next, the %arguments 
 observations made above yield that we can apply %the assumptions for 
Proposition 4.3.15 of ibid. %are fulfilled; thus we have 
 to obtain $f^*x_*\cong X_{K*}f_X^*$. Combining this statement with the %aforementioned
 isomorphisms mentioned above
 we obtain $f^*(\mgrc(X))\cong x_{K*} x_{K}^! R_{\spe K} \cong \mgrc(X_{K})$ indeed. 

\ref{ie3}. %Perfect closure??
$N$ belongs to $ \dmer^{\thomr \le 0}$ if and only if for any function field $k'/k$, any presentation $\spe k'=\prli X_j$ for $X_j\in \sv$, $r\ge 0$,  and $l< 0$ we have  $\inli_j \dmer(\mgr(X_j)\lan r\ra [l-r],N)=\ns$; see Corollary 5.2, %\S5.7(1),  and 
 \S1.18, and Theorem 3.7 of \cite{degmod}.\footnote{Alternatively, one may apply Corollary 2.3.12 and Theorem 3.3.1 of  \cite{bondegl}; see also Definitions 1.3.10  and 3.2.3, and the formula (2.3.4.a) of ibid.}

Next we recall that the motives $\mgr(X_j)$ can also be computed as $x_{j\#}R_{X_j}=x_{j\#}{x_j}^*R_{\spe k}$, where $x_j:X_j\to \spe k$ is the structure morphisms and $x_{j\#}$ is left adjoint to ${x_j}^*$; see the formula (8.5.3) of \cite{cdint} (yet we omit L's in the notation of loc. cit.). Thus $ \dmer(\mgr(X_j)\lan r\ra [l-r],N)\cong \dmcdh(X_j,R)(R_{X_j} \lan r\ra [l-r],{x_j}^*N)$. We can pass to the limit in this isomorphism using the aforementioned continuity property (see Definition 2.5 of ibid.) %  Thus we 
 to obtain  $\inli \dmcdh(X_j,R)(R_{X_j}  \lan r\ra [l-r],{x_j}^*N)\cong \dmcdh(\spe k',R)(R_{\spe k'}\lan r\ra [l-r],N_{\spe k'})$. The latter group is isomorphic to $\chowm_{r}(N_{k'},R,l-r)$ since we can replace $k'$ by its perfect closure; see Proposition 8.1 of \cite{cdint}.

It suffices to verify the vanishing in question for $r=0$. This statement easily follows from Proposition 5.2.6(8) and Remark 5.2.7(7) of \cite{bgn}. Moreover, it can be easily deduced from the well-known Theorem 4.19 of \cite{3} %Proposition 1.4.2(1) of \cite{bger}  the obvious 
 along with the fact that $R\lan n\ra[-n]$ is a retract of $\mgr(\mathbb{G}_m^n)$ (where $\mathbb{G}_m=\afo\setminus \ns$). % along with the well-known Theorem 4.19 of \cite{3}.   
%Since the $t$-structure $\thomr$ is non-degenerate (see Corollary 5.2 of \cite{degmod} once again), it suffices to check that the $t$-homology of 
\begin{comment} 
We recall that we can study $\thomr$ using the results of \cite{bondegl}; see Corollary 2.3.12 and Example 2.3.13(1) of ibid. Note that the corresponding {\it dimension function} $\delta$ over $k$ sends the spectrum of a function field $k'$ over $k$ into the transcendence degree of $k'/k$.
Applying Theorem 3.3.1 of ibid. we obtain that  $N$ belongs to $ \dmer^{\thomr \le 0}$ if and only if for any  function field $k'/k$, any presentation $\spe k'=\prli X_j$ for $X_j\in \var$, $r\ge 0$,  and $l< 0$ we have  \begin{equation}\label{ebm}  \inli_j \dmcdh(\spe k,R)
(x_{j!}R_{X_j}\lan r+\delta(k')\ra [l-r],N)=\ns,\end{equation} where $x_j:X_j\to \spe k$ are the structure morphisms; see  Definitions 1.3.10  and 3.2.3 of ibid. Next, we can certainly assume that all these $X_j$ are smooth and of dimension $\delta(k')$ over $k$; hence $x_{j!}R_{X_j}\lan \delta(k')\ra\cong \mgr(X_j)$ (see the formula (2.3.4.a) of ibid.). 
; see Example 2.3.13, Definition 3.2.3,  and Theorem 3.3.1 of \cite{bondegl}.
\end{comment}
\end{proof}

\end{document}